\documentclass{lmcs}

\usepackage[utf8]{inputenc}

\usepackage[style=authoryear,maxbibnames=99,maxcitenames=4,backend=biber]{biblatex}
\usepackage[english]{babel}
\usepackage[T1]{fontenc}
\usepackage{csquotes}

\usepackage{tikz}
\usetikzlibrary{cd}
\usetikzlibrary{decorations.pathmorphing}

\usepackage{stmaryrd}
\usepackage{amsmath}
\usepackage{amsthm}
\usepackage{amssymb}
\usepackage{thmtools}
\usepackage{mathtools}
\usepackage{mathpartir}
\usepackage{environ}
\usepackage{bm}
\usepackage{bbm}
\usepackage{hyperref}
\usepackage{xspace}
\usepackage{stackengine}

\hypersetup{
  breaklinks=true,
  colorlinks=true
  }

\theoremstyle{plain}\newtheorem{construction}[thm]{Construction}

\input{macros.tex}

\newcommand{\CsSet}{\mathbf{sSet}}
\newcommand{\CcSet}{\mathbf{cSet}}
\newcommand{\CcCat}{\mathbf{cCat}}

\newcommand{\wcom}{\mathsf{wcom}}
\newcommand{\wcoe}{\mathsf{wcoe}}
\newcommand{\wcoh}{\mathsf{wcoh}}

\newcommand{\ext}{\mathsf{ext}}

\newcommand{\Cof}{\mathrm{Cof}}
\newcommand{\fib}{\mathsf{fib}}
\newcommand{\coffib}{{\Cof\text{-}\mathsf{fib}}}

\newcommand{\HasWCom}{\mathsf{HasWCom}}
\newcommand{\HasWCoe}{\mathsf{HasWCoe}}
\newcommand{\HasExt}{\mathsf{HasExt}}

\newcommand{\Glue}{\mathsf{Glue}}
\newcommand{\glue}{\mathsf{glue}}

\newcommand{\Iso}{\mathsf{Iso}}

\newcommand{\EqHom}{\mathsf{EqHom}}
\newcommand{\El}{\mathsf{El}}

\newcommand{\Th}{\CT}
\newcommand{\ua}{\mathsf{ua}}

\newcommand{\idToPath}{{\mathsf{id}\text{-}\mathsf{to}\text{-}\mathsf{path}}}
\newcommand{\CScone}{\mathbf{Scone}}

\newcommand{\Lift}{\mathsf{Lift}}
\newcommand{\lift}{\mathsf{lift}}

\newcommand{\funr}{\overrightarrow{\mathsf{fun}}}
\newcommand{\funl}{\overleftarrow{\mathsf{fun}}}
\newcommand{\isRefl}{\mathsf{isRefl}}

\newcommand{\Path}{\mathsf{Path}}
\newcommand{\Loop}{\mathsf{Loop}}
\newcommand{\ReflLoop}{\mathsf{ReflLoop}}

\newcommand{\RelEquiv}{\mathsf{RelEquiv}}

\newcommand{\Colon}{\mathrel{::}}

\begin{document}

\title{Strict Rezk completions of models of HoTT and homotopy canonicity}

\author[R.~Bocquet]{Rafaël Bocquet}
\email{bocquet@inf.elte.hu}

\begin{abstract}
  We give a new constructive proof of homotopy canonicity for homotopy type theory (HoTT).
  Canonicity proofs typically involve gluing constructions over the syntax of type theory.
  We instead use a gluing construction over a ``strict Rezk completion'' of the syntax of HoTT.
  The strict Rezk completion is specified and constructed in the topos of cartesian cubical sets.
  It completes a model of HoTT to an equivalent model satisfying a completeness condition, providing an equivalence between terms of identity types and cubical paths between terms.
  This generalizes the ordinary Rezk completion of a $1$-category.
\end{abstract}

\maketitle

\section{Introduction}

Voevodsky conjectured~\parencite{UnivalentFoundationsProject}
that the extension of Martin-Löf Type Theory (MLTT) with his univalence axiom remains constructive.
More precisely, homotopy canonicity for Homotopy Type Theory (HoTT) is the statement that
any closed term of the type of natural numbers in the syntax of HoTT is identifiable with a numeral,
where the identification is witnessed by some closed term of the identity type.

Strict canonicity for Martin-Löf Type Theory can be proven by a model construction known as categorical gluing.
It involves gluing together the syntax of MLTT with the category of sets.
The gluing is specified by the global sections functor, which assigns to every syntactic context its set of closing substitutions (gluing along the global sections functor is also called sconing).
For proofs of homotopical properties of the syntax, such as homotopy canonicity for HoTT, the set-valued global sections functor should be replaced by a homotopical global sections functor, valued in $\infty$-groupoids (or spaces).
However, coherence issues arise, as the syntax has a strict underlying $1$-category $\CS$, while $\infty$-groupoids form an $\infty$-category (perhaps presented by some $1$-category, such as simplicial sets with the Kan-Quillen model structure).

For any syntactic context $\Gamma \in \CS$, one wishes to define an $\infty$-groupoid of closing substitutions into $\Gamma$.
Its set of objects should be the set $\CS(1_\CS,\Gamma)$ of closing substitutions, but the higher cells should be given by iterated identity types $\CS(1_\CS,\Id_\Gamma(-,-))$, \etc
Unfortunately, defining these spaces in a way that is strictly functorial in $\Gamma$, \eg a functor $\CS \to \CsSet$, does not seem possible as a direct construction.

\Textcite{SattlerHomotopyCanonicityHoTT} obtained a proof of homotopy canonicity for HoTT, although the details of the proof have not been made public yet.
Their strategy is to present the homotopical global sections functor using a span
\[ \CS \la \mathsf{Fr}(\CS) \to \CsSet. \]
Here $\mathsf{Fr}(\CS)$ is the frame model over $\CS$, the homotopical inverse diagram model indexed by the semi-simplex category $\Delta_+$.
The frame model extends the syntax with more data, and this additional data allows for the definition of a strict functor $\mathsf{Fr}(\CS) \to \CsSet$.
The map $\mathsf{Fr}(\CS) \to \CS$ is a weak equivalence of models, ensuring that the span morally corresponds to a functor $\CS \to \CsSet$.
Using simplicial sets leads to a non-constructive proof of homotopy canonicity, but a constructive proof can be achieved by gluing along some more complex functor $\mathsf{Fr}(\CS) \to \CcSet_{\mathsf{dM}}$ into (De Morgan) cubical sets.

In this paper, we propose another way to solve the issue of the definition of a homotopical global sections functor.
We work internally to the topos $\CcSet$ of cartesian cubical sets.
In the internal language of this topos, we have a notion of fibrant set; the fibrant sets can be seen as $\infty$-groupoids.
This topos has been equipped with the structure of a model of HoTT with universes classifying the fibrant sets by~\textcite{SyntaxModelsCartCTT}.
There is an internal copy of the syntax $\CS$ of HoTT;
its components (sets of contexts, substitutions, \etc) are fibrant but have the ``wrong'' homotopy types
(they are $0$-truncated and ignore the ambient cubical setting).
We will define another internal model $\overline{\CS}$ with the following properties:
\begin{itemize}
\item The components of $\overline{\CS}$ are fibrant; \ie they can be seen as $\infty$-groupoids.
\item There is a morphism $i : \CS \to \overline{\CS}$ of models of HoTT.
  Furthermore, after externalization (restriction to the empty cubical context), $i$ becomes a weak equivalence of models of HoTT.
\item The model $\overline{\CS}$ is \defemph{complete}, meaning that its components have the correct homotopy types;
  in particular we have equivalences
  \[ (x \sim y) \simeq \overline{\CS}.\Tm(\Gamma,\Id_A(x,y)), \]
  where $(x \sim y)$ is the set of paths between $x$ and $y$ in $\overline{\CS}.\Tm(\Gamma,A)$.
  More precisely, the (fibrant) set
  \[ (y : \overline{\CS}.\Tm(\Gamma,A)) \times \overline{\CS}.\Tm(\Gamma,\Id_A(x,y)) \]
  should be contractible for any term $x$.
\end{itemize}

Once this model $\overline{\CS}$ is constructed,
we have, still internally to cartesian cubical sets,
a well-behaved homotopical global sections functor,
sending a syntactic context $\Gamma$ to the fibrant set $\overline{\CS}(1,i(\Gamma))$.
Homotopy canonicity for $\CS$ then follows from a standard gluing argument.
In the strict canonicity proof for MLTT, a closed type $A$ is sent to a unary logical predicate
\[ \sem{A} : \CS.\Tm(1,A) \to \SSet. \]
For the homotopy canonicity proof, we instead interpret a closed type $A$ as a logical predicate
\[ \sem{A} : \overline{\CS}.\Tm(1,i(A)) \to \SSet^\fib \]
valued in fibrant sets.

We call $\overline{\CS}$ the \textbf{strict Rezk completion} of $\CS$.
Indeed, its specification is closely related to the specification of Rezk completions of categories~\parencite{UnivalentCategories}.
If $\CC$ is a category in HoTT (meaning that the categorical laws hold up to identification), then its Rezk completion is a category $\overline{\CC}$ satisfying the following properties:
\begin{itemize}
\item There is a weak equivalence $F : \CC \to \overline{\CC}$ (a functor that is essentially surjective and fully faithful).
\item The category $\overline{\CC}$ is \defemph{complete} (or univalent):
  objects of $\overline{\CC}$ have the correct homotopy types; in particular we have equivalences
  \[ (x \sim y) \simeq \Iso_{\overline{\CC}}(x,y) \]
  between identifications in $\Ob_{\overline{\CC}}$ and isomorphisms.
  This is alternatively expressed by asking for the contractibility of $(y : \Ob_{\overline{\CC}}) \times \Iso_{\overline{\CC}}(x,y)$.
\end{itemize}

Strict Rezk completion differs from the ordinary Rezk completion for categories.
The ordinary Rezk completion can be specified (and constructed) fully in HoTT; the categorical and functorial laws are then expressed using identifications in HoTT.
The strict Rezk completion cannot be specified in HoTT: it needs a notion of strict equality (available in cubical sets, and more generally in models of two-level type theory~\parencite{2LTT}).
This is crucial, because it seems that the notion of model of HoTT cannot be expressed without either strict equalities,
or an infinite tower of additional coherence data.
\Textcite{KrausInftyCwfs} explains that internally to HoTT, neither set-truncated nor wild models are well-behaved.

Strict Rezk completions can be specified not only for models of HoTT, but also for the categories of algebras of generalized algebraic theories with a suitable homotopy theory.
In this paper, we consider the case of $1$-categories with the canonical model structure, and models of HoTT with an algebraic variant of the left semi-model structure introduced by~\textcite{HomotopyTheoryTTs}.
We leave generalization to other generalized algebraic theories to future work;
we expect that our construction of the strict Rezk completion should work for
generalized algebraic theories whose category of models is equipped with a
``pseudo-cylindrical left semi-model structure''.

The main idea behind the construction of the strict Rezk completion is to reformulate the notion of completeness in a way that interacts well with the cubical structure.
Completeness is defined using contractibility conditions, and in cubical presheaf models, the notion of contractibility can be expressed in two ways:
\begin{enumerate}
\item\label{itm:contr_1}
  Using the usual definition from HoTT:
  \[ \isContr(X) \triangleq (x:X) \times (\forall y \to x \sim y). \]
\item\label{itm:contr_2}
  Using the cubical structure:
  contractibility corresponds to trivial fibrancy;
  a set is trivially fibrant if any partial element can be extended to a total element.
\end{enumerate}

For fibrant sets,
both definitions are logically equivalent.
For sets that are not yet known to be fibrant,
(\ref{itm:contr_2}) behaves better.

Our definition of the strict Rezk completion relies on the notion of trivial fibrancy: the strict Rezk completion $\overline{\CC}$ of a category $\CC$ is defined as the free extension of $\CC$ by extension structures for the sets $(y : \Ob_{\overline{\CC}}) \times \Iso_{\overline{\CC}}(x,y)$.
We then have to prove that this category has fibrant components and that the externalization of the functor $i : \CC \to \overline{\CC}$ is a weak equivalence.
This generalizes a construction by~\textcite{FoundationSyntheticAlgGeo} of the propositional truncation without homogeneous fibrant replacement in cubical sets.
In fact, their construction of the propositional truncation can be seen as the simplest example of a strict Rezk completion (for the category of sets, equipped with a homotopy theory presenting the propositions).

\subsection*{Related work}

This work builds upon the axiomatic development the semantics of cubical type theories~\parencite{CCHM17} in the internal language of toposes~\parencite{AxiomsCubicalTTInTopos,InternalUniversesModelsHOTT,SyntaxModelsCartCTT,UnifyingCubicalModels}.
The definition of the strict Rezk completion using trivial fibrancy is related to $\Glue$-types and to the equivalence extension property.
The proof of fibrancy of the components of the strict Rezk completion is similar to the proof of fibrancy of the universes in the cubical presheaf models.

We also rely on the left-semi model structure on categories of models of type theories from~\cite{HomotopyTheoryTTs}, and on homotopical inverse diagram models~\parencite{HomotopicalInverseDiagrams}.
This left semi-model structure presents an $\infty$-category of models.
The $\infty$-type theories of~\textcite{InftyTypeTheories} have $\infty$-categories of space-valued models, but relating these space-valued models to the set-valued models of a $1$-type theory is not easy.
The strict Rezk completion is a way to relate set-valued models and space-valued models of type theory, without using $\infty$-categorical tools.

Some canonicity and normalization results have previously been obtained for type theories with univalent universes.
A homotopy canonicity result for a $1$-truncated type theory with a univalent universe of sets has been obtained by~\textcite{UnivalenceInverseDiagramsHomotopyCanonicity}.
Strict canonicity for cubical type theory was first proven by~\textcite{HuberCanonicity}.
\Textcite{CanonicityCubical} have used gluing constructions to prove homotopy and strict canonicity for cubical type theory.
Normalization for cubical type theory has been proven by~\textcite{NormalizationCubical}.
For cubical type theory, taking a strict Rezk completion of the syntax is not needed, because the cubical structure of the syntax automatically endows it with the correct higher dimensional structure.

\subsection*{Outline}

We begin in~\autoref{sec:background} by reviewing the axiomatization of~\textcite{SyntaxModelsCartCTT} of the cartesian cubical set model in the internal language of a topos.
We use the notion of weak composition structure due to~\textcite{UnifyingCubicalModels}.
In~\autoref{sec:completion_categories} we specify and construct the strict Rezk completion for $1$-categories.
The goal to show the main ideas of this work in a relatively simple setting; the construction of the strict Rezk completion for models of HoTT will follow the same structure.

In~\autoref{sec:hott} we detail the semantics of our variant of HoTT, which has a cumulative hierarchy of univalent universes, $\Sigma$-types, $\Pi$-types, booleans, unit-types, empty-types and $W$-types.
We also define part of the homotopy theory of models of HoTT, following~\textcite{HomotopyTheoryTTs}.
In particular, path models and reflexive-loop models, which are instances of homotopical inverse diagram models~\parencite{HomotopicalInverseDiagrams}, play an important role.
Then in~\autoref{sec:completion_hott} we specify and construct the strict Rezk completion for models of HoTT.

Finally, in~\autoref{sec:homotopy_canonicity}, we prove homotopy canonicity for HoTT, relying on the strict Rezk completion of the syntax.

\subsection*{Agda formalization}

\begin{sloppypar}
  The constructions of path and reflexive-loop models of HoTT,
  appearing in~\autoref{ssec:path_and_reflloop_models},
  have been partly formalized in Agda.
  The formalization is available at~\url{https://rafaelbocquet.gitlab.io/Agda/20230925_StrictRezkCompletionsAndHomotopyCanonicity/},
  and in the ancillary files attached to the arXiv submission of the paper.
\end{sloppypar}

\subsection*{Acknowledgments}
I thank Christian Sattler for discussions about this work.
His comments on an early draft of this paper have led to simplifications of the proof methods.


\section{Background: cartesian cubical sets}\label{sec:background}

Most of our development takes place internally to the topos $\CcSet$ of presheaves over the cartesian cube category $\square$.
We use the notion of weak composition structure from~\textcite{UnifyingCubicalModels},
but we also rely on diagonal cofibrations,
so our development is valid in cartesian cubical sets but not in De Morgan cubical sets.

We recall the cubical structure that is available in the internal language of cartesian cubical sets.

\subsection{Interval and cofibrations}

There is an \defemph{interval} $\MI : \SSet$,
with two points $0,1 : \MI$
such that $0 \neq 1$.

There is set $\Cof$ of \defemph{cofibrations}.
Every cofibration $\alpha$ has an associated proposition $[\alpha] : \Omega$.

Cofibrations are closed under \defemph{interval equalities} $(i = j)$,
binary and nullary \defemph{conjunctions} ($(\alpha \wedge \beta)$ and $\top$),
binary and nullary \defemph{disjunctions} ($(\alpha \vee \beta)$ and $\bot$)
and \defemph{quantification over the interval} $(\forall i:\MI. \alpha(i))$.

Moreover, we assume that cofibrations are inductively generated by interval equalities, binary conjunctions and binary disjunctions.
We will need to make use of that fact in~\autoref{ssec:ext_finitary},
to justify that the notion of partial element is finitary.

Diagonal cofibrations (interval equalities that are not of the form $(i=\varepsilon)$ for $\varepsilon\in\{0,1\}$)
are only used in the proofs of~\autoref{prop:cat_haswcoe_tfib} and~\autoref{prop:hott_haswcoe_tfib}.

Given a cofibration $\alpha$ and a set $X$,
an element $x : [\alpha] \to X$ is said to be a \defemph{partial element} of $X$.
In that case, a total element $x' : X$ is said to extend $x$ if $[\alpha]$ implies that $x = x'$.
We write $\{ X \mid \alpha \hra x \}$ for the set of total elements of $X$ extending $x$.

A partial element in $[\alpha] \to X$
may be written $[\alpha \mapsto x]$.
When $\alpha$ is a disjunction $(\phi \vee \psi)$,
we may write $[\phi \mapsto x_\phi, \psi \mapsto x_\psi]$
for the unique element of $[\phi \vee \psi] \to X$
that restrict to $x_\phi$ under $\phi$ and $x_\psi$ under $\psi$,
assuming $[\phi \wedge \psi] \to (x_\phi = x_\psi)$.
We write $[]$ for the unique element of $[\bot] \to X$.

Let $A : \MI \to \SSet$ be
a line of sets with two points $x_0 : A(0)$ and $x_1 : A(1)$.
A \defemph{dependent path} $p : x_0 \sim_A x_1$
is a map $p : (i : \MI) \to A(i)$
such that $p(0) = x_0$ and $p(1) = x_1$.
A non-dependent path in $A : \SSet$
is a path over the constant line $(\lambda \_ \mapsto A)$.

We will use the symbols
$(\sim)$ for paths,
$(\simeq)$ for equivalences,
and $(\cong)$ for isomorphisms.

\subsection{Global elements}

When $X$ is a global set in the internal language of $\CcSet$,
we write $1^\ast_\square(X)$
for the external set of global elements of $X$.
This can also be identified
with the evaluation of $X$
at the terminal object $1_\square$ of the cartesian cube category.
The interaction between internal and external reasoning could also be expressed using modalities,
\eg either crisp type theory~\parencite{ShulmanCrisp}
or the dependent right adjoint~\parencite{DependentRightAdjoints}
corresponding to the inverse image $1^\ast_\square : \CcSet \to \CSet$.

We also rely on
$1^\ast_\square$ being a functor preserving finite limits;
in particular it acts on algebras and homomorphisms of any essentially algebraic theory,
\eg if $\CC$ is an internal category,
then $1^\ast_\square(\CC)$ is an external category.

We also assume that $1^\ast_\square(\Cof) \cong \{\true,\false\}$,
and that the corresponding map $[-] : \{\true,\false\} \to 1^\ast_\square(\Omega)$ selects the propositions $\top$ and $\bot$.
Indeed, the only sieves over the terminal object $1_\square$ of the cartesian cube category are $\top$ and $\bot$.

\subsection{Tinyness of the interval}\label{ssec:int_tiny}

The interval is tiny, which means that
the exponentiation functor $(-)^\MI : \CcSet \to \CcSet$
has a right adjoint $(-)_\MI : \CcSet \to \CcSet$.
This is an external statement; the right adjoint cannot be accessed internally.
Most of the time, we don't use the right adjoint directly.
Instead, we use the following consequence:

\begin{lem}[{\cite[Lemma~2.2]{CanonicityCubical}}]\label{lem:int_tiny}
  Let $A$ be a global set and $B$ be a global family over $A^\MI$.
  Then we have a global family $B_\MI$ over $A$ with a bijection of global elements
  \[ 1^\ast_\square((a : X) \to B_\MI(f(a))) \cong 1^\ast_\square((a : X^\MI) \to B(f \circ a)) \]
  natural in global $f : X \to A$.

  The construction $(-)_\MI$ may be chosen so that:
  \begin{enumerate}
  \item If $B$ is $i$-small, then so is $B_\MI$.
  \item The induced isomorphism $(\lambda a \to B(f \circ a))_\MI \cong (B_\MI \circ f)$ is the identity.
    \qed
  \end{enumerate}
\end{lem}

Remark that both $(-)^\MI$ and $(-)_\MI$ are right adjoint functors,
thus preserving all limits.
This entails that the endo-adjunction descends to an endo-adjunction $((-)^\MI \dashv (-)_\MI)$
on the category of cubical algebras of any essentially algebraic theory.
Moreover, the actions of $(-)^\MI$ and $(-)_\MI$ on algebras can be computed sortwise.

Indeed, from the point of view of functorial semantics in finitely complete categories,
a cubical $\Th$-algebra is a left exact functor $\Th \to \CcSet$.
The functors $(-)^\MI$ and $(-)_\MI$ then act by post-composition.

\subsection{Kan operations}

We review the definitions and properties of the Kan operations, which are used to define the notions of fibrancy and trivial fibrancy.
We use the notion of weak composition structure due to~\textcite{UnifyingCubicalModels}.

\begin{defi}
  A \defemph{weak composition structure} for $A : \MI \to \SSet$ consists of:
  \begin{itemize}
  \item For every $r,s : \MI$, $\alpha : \Cof$, $t : [\alpha] \to (s:\MI) \to A(s)$ and $b : A(r)$ such that $[\alpha] \to t(r) = b$,
    there is an element
    \[ \wcom_A^{r \to s}(t,b) : A(s) \]
    such that $[\alpha] \to \wcom_A^{r \to s}(t,b) = t(s)$.
  \item There is a family of paths
    \[ \underline{\wcom}^{r}_A(t,b) : \wcom_A^{r \to r}(t,b) \sim b \]
    such that $[\alpha] \to \underline{\wcom}^{r}_A(t,b) = (\lambda \_ \mapsto b)$.
    \qedhere
  \end{itemize}
\end{defi}

This defines a global family $\HasWCom : \SSet^\MI \to \SSet$.
We obtain a global family $\HasWCom_\MI : \SSet \to \SSet$ by~\autoref{lem:int_tiny}.
Elements of $\HasWCom_\MI(X)$ are called fibrancy structures over $X$, and sets equipped with a fibrancy structure are called fibrant sets.

\autoref{lem:int_tiny} provides a global map
\[ \epsilon : (X : \MI \to \SSet) (X^\fib : (i : \MI) \to \HasWCom_\MI(X(i))) \to \HasWCom(X), \]
which is essentially the counit of the adjunction $((-)^\MI \dashv (-)_\MI)$.
This justifies using the notations $\wcom_{i.X(i)}$ and $\underline{\wcom}_{i.X(i)}$ whenever $X$ is a line of fibrant sets.

The notion of weak composition structure can be reformulated in terms of limits and split surjections,
this will be used in the proof of~\autoref{lem:cat_fib_components_map}.
\begin{lem}\label{lem:wcom_as_surj}
  A line $A : \MI \to \SSet$ has a weak composition structure if and only for every $r:\MI$ and $\alpha : \Cof$ the map
  \begin{alignat*}{1}
    & (b : A(r))
      \times (w : (s:\MI) \to \{ A(s) \mid [\alpha \wedge (s=r)] \hra b \})
      \times \{ \underline{w} : w(r) \sim b \mid [\alpha] \hra (\lambda \_{} \mapsto b) \} \\
    & \quad \to (b : A(r)) \times (t : [\alpha] \to (s:\MI) \to \{ A(s) \mid [s=r] \hra b \}), \\
    & (b, w, \underline{w}) \mapsto (b, \lambda s \mapsto w(s)).
  \end{alignat*}
  is a split surjection.
  \qed
\end{lem}

\begin{defi}
  An \defemph{extension structure} for a set $X$ is the data, for every $\alpha : \Cof$ and partial element $x : [\alpha] \to X$, of a total element
  \[ \ext_X(x) : \{ X \mid \alpha \hra x \}.
    \tag*{\qedhere{}}
  \]
\end{defi}
We write $\HasExt$ for the family of extension structures.
A set equipped with an extension structure is said to be trivially fibrant.

In cubical models, a fibrant set is contractible if and only if it has an extension structure.
For non-fibrant sets, or sets that are not yet known to be fibrant, extension structure are better behaved than the usual definition of contractibility.
\begin{prop}[{\cite[Lemma~5]{CCHM17}}]\label{prop:fibrant_contractibility_iff_ext}
  For any set $X$, there is a logical equivalence
  \[ \HasExt(X) \lra (\HasWCom_\MI(X) \times \isContr(X)), \]
  where
  \[ \isContr(X) = (x : X) \times ((y : X) \to (x \simeq y)). \]

  In other words, a set is trivially fibrant if and only if it is fibrant and contractible.
\end{prop}
\begin{proof}
  We prove both implications.
  \begin{description}
  \item[$(\La)$]
    Assume that $X$ is fibrant and contractible.
    We equip $X$ with an extension structure.

    Take a partial element $x : [\alpha] \to X$.
    Write $x_0 : X$ for the center of contraction of $X$.
    Since $X$ is contractible, we have a partial path $p : [\alpha] \to x_0 \simeq x$.
    
    Now $\ext_X(x) \triangleq \wcom_X^{0 \to 1}([\alpha\ i \mapsto p(i)], x_0)$ is a total element extending $x$.
  \item[$(\La)$]
    Assume that $X$ has an extension structure.

    We first prove that $X$ is contractible.
    We can find a center of contraction $\mathsf{ext}([])$ by extending the empty partial element.
    Given $x : X$, $y : X$ and $i : \MI$, we can define an element
    $p(i) = \mathsf{ext}([(i = 0) \mapsto x, (i = 1) \mapsto y])$.
    Then $p : \MI \to X$ is a path between $x$ and $y$, as needed

    We now prove that $X$ is fibrant, by defining a global map
    \[ (X : \SSet) \times \HasExt(X) \to \HasWCom_\MI(X). \]
    By~\autoref{lem:int_tiny}, it suffices to construct a map
    \[ (X : \SSet^\MI) \times (\ext_{X(-)} : (i:\MI) \to \HasExt(X(i))) \to \HasWCom(X). \]
    We pose
    \begin{alignat*}{1}
      & \wcom_X^{r \to s}(t,b) \triangleq \ext_{X(s)}([\alpha \mapsto t(s)]), \\
      & \underline{\wcom}_X^{r}(t,b,i) \triangleq \ext_{X(r)}([\alpha \mapsto b,\ (i=0) \mapsto \wcom_X^{r \to r}(t,b),\ (i=1) \mapsto b]).
    \end{alignat*}
    One can check that the necessary boundary conditions are satisfied.
  \end{description}
\end{proof}

\subsection{Fibrancy from reflexive graphs and pseudo-reflexive graphs}

For a universe level $n$, the universe $\SSet^\fib_n$ of $n$-small fibrant sets is defined as
\[ \SSet^\fib_n \triangleq (A : \SSet_n) \times \HasWCom_\MI(A). \]

As shown by~\textcite{UnifyingCubicalModels}, it is univalent and closed under $\Pi$-types, $\Sigma$-types, Path-types, $\Glue$-types, \etc

The proof of fibrancy of the universe factors through the proof of the equivalence extension property,
which follows from the construction of $\Glue$-types.
The equivalence extension property says that the set
\[ (B : \SSet^\fib) \times \Equiv(A,B) \]
is trivially fibrant for any $A : \SSet^\fib$
\ie that we have an extension structure for equivalences with a fixed left endpoint.

The fibrancy of the universe then follows from the following lemma,
instantiated at $A = \SSet^\fib$, $E_A = \Equiv$ and $r(A) = \id_A$.
\begin{lem}
  Take a global reflexive graph with
  vertices $A : \SSet$,
  edges $E_A : V \to V \to \SSet$ and
  reflexivity map $r : (a : A) \to E_A(a,a)$.
  Assume that the following two conditions are satisfied:
  \begin{description}
  \item[$A$ has a weak coercion structure]
    For every line $a : \MI \to V$, we have operations
    \[ \wcoe_a^{r \to s} : E_A(a(r),a(s)) \]
    and
    \[ \wcoh_a^{r} : (\wcoe_a^{r \to r} = r(a)). \]
  \item[$A$ is homotopical]
    For every $a_1 : A$, the set
    \[ (a_2 : A) \times E_A(a_1,a_2) \]
    is trivially fibrant.
  \end{description}
  Then $V$ is fibrant.
\end{lem}
\begin{proof}
  Proof omitted, we prove a more general version in~\autoref{lem:fibrancy_from_reflgraph}.
\end{proof}

We give a generalization of that proof of fibrancy.
A pseudo-reflexive graph object in a category is a diagram indexed by the category
\[ \begin{tikzcd}
    & R \ar[r]
    \ar[r, "p_e"]
    & E
    \ar[r, "p_1", shift left]
    \ar[r, "p_2"', shift right]
    & V,
  \end{tikzcd} \]
with $p_1 \circ p_e = p_2 \circ p_e$.

This category is an inverse replacement of the indexing category for reflexive graphs.
The object $V$ corresponds to vertices, the object $E$ corresponds to edges, and the object $R$ corresponds to reflexive loops.
A reflexive graph is exactly a pseudo-reflexive graph $(V,E,R)$ such that $V = R$.

A pseudo-reflexive graph in $\CSet$ is a triple $(A,E_A,R_A)$, where
\begin{alignat*}{1}
  & A : \SSet, \\
  & E_A : (a_1 : A) (a_2 : A) \to \SSet, \\
  & R_A : (a : A) (a_e : E_A(a,a)) \to \SSet.
\end{alignat*}
Throughout the paper, we will use the same notations when quantifying over the elements of a pseudo-reflexive graph
($a:A$, $a_e : E_A(a_1,a_2)$ and $a_r : R_A(a,a_e)$).
We may implicitly quantify over some elements,
\eg quantifying over $a_e : E_A(a_1,a_2)$ may implicitly quantify over $a_1,a_2 : A$.

\begin{lem}\label{lem:fibrancy_from_reflgraph}
  Assume given the data of a global dependent pseudo-reflexive graph $(B,E_B,R_B)$ over a base pseudo-reflexive graph $(A,E_A,R_A)$:
  \begin{alignat*}{1}
    & A : \SSet, \\
    & E_A : A \to A \to \SSet, \\
    & R_B : (a : A) (a_e : E_A(a,a)) \to \SSet, \\
    & B : A \to \SSet, \\
    & E_B : E_A(a_1,a_2) \to B(a_1) \to B(a_2) \to \SSet, \\
    & R_B : R_A(a,a_e) \to (b : B(a)) (b_e : E_B(a_e,b,b)) \to \SSet.
  \end{alignat*}

  We also assume (globally) the following conditions:
  \begin{description}
  \item[$A$ and $B$ have weak coercion structures]
    For any $a : \MI \to A$, we have
    \[ \wcoe_a^{r \to s} : E_A(a(r),a(s)) \]
    and
    \[ \wcoh_a^{r} : R_A(a,\wcoe_a^{r \to r}) \]
    and for any $b : (i:\MI) \to B(a(i))$ we have
    \[ \wcoe_b^{r \to s} : E_B(\wcoe_a^{r \to s},b(r),b(s)) \]
    and
    \[ \wcoh_b^{r} : R_B(\wcoh_a^{r}, b, \wcoe_b^{r \to r}). \]

  \item[{$B$ is homotopical}]
    For any $b_1 : B(a_1)$ and $a_e : E_A(a_1,a_2)$, the set
    \[ (b_2 : B(a(s))) \times (b_e : E_B(a_e,b,b_e)) \]
    is trivially fibrant, and that for any $b : B(a)$ and $a_r : R_A(a,a_e)$, the set
    \[ (b_e : E_B(a_e,b,b)) \times (b_r : R_B(a_r,b,b_e)) \] is trivially fibrant.
  \end{description}
  
  Then $B$ is a family of fibrant sets.
\end{lem}
\begin{proof}
  By~\autoref{lem:int_tiny}, it suffices to construct a global element of $(a : A^\MI) \to \HasWCom(\lambda i \mapsto B(a(i)))$.
  
  Take $r : \MI$, a cofibration $\alpha : \Cof$ and elements $t : [\alpha] \to (s:\MI) \to B(a(s))$ and $b : B(a(r))$ such that $[\alpha] \to t(r) = b$.
  We use the homotopicality of $B$ to extend some partial elements.

  Given $s : \MI$, we let $w(s)$ be an element of $(b_2 : B(a(s))) \times (b_e : E_B(\wcoe_a^{r \to s},b,b_e))$ extending the partial element
  \[ [\alpha \mapsto (t(s),\wcoe_t^{r \to s})]. \]

  We let $d$ be an element of $(b_e : E_B(\wcoe_a^{r \to r},b,b)) \times (b_r : R_B(\wcoh_a^{r},b,b_e))$ extending
  \[ [\alpha \mapsto (\wcoe_t^{r \to r}, \wcoh_t^r)]. \]
  
  Given $i : \MI$, we let $\underline{w}(i)$ be an element of $(b_2 : B(a(r))) \times (b_e : E_B(\wcoe_a^{r \to r},b,b_e))$ extending
  \[ [\alpha \mapsto (t(r),\wcoe_t^{r \to r}),\ (i=0) \mapsto w(s),\ (i=1) \mapsto (b,d.1)]. \]

  We can then define a weak composition structure as follows:
  \begin{alignat*}{1}
    & \wcom_{B(a(-))}^{r \to s}(t,b) = w(s).1, \\
    & \underline{\wcom}_{B(a(-))}^{r}(t,b,i) = \underline{w}(i).1.
      \tag*{\qedhere}
  \end{alignat*}
\end{proof}

\subsection{Propositional truncation without homogeneous fibrant replacement}

The propositional truncation of a fibrant set can be defined as a higher inductive type.
The semantics of higher inductive types in $\CcSet$ involves a set freely generated by the constructors of the higher inductive type and additional constructors ensuring fibrancy (a form of fibrant replacement).

As observed by~\textcite{FoundationSyntheticAlgGeo}, the propositional truncation can actually be defined without fibrant replacement, when trivial fibrancy is used to express propositionality (recall that $\mathsf{isProp}(X) \lra (X \to \mathsf{isContr}(X))$).

We recall how to perform this construction (in the simpler case of global sets).

\begin{thm}[\cite{FoundationSyntheticAlgGeo}]\label{thm:trunc_wo_hcom}
  Let $X$ be a global fibrant set and $\overline{X}$ be the set freely generated by a map $i : X \to \overline{X}$ and by an element of $\overline{X} \to \HasExt(\overline{X})$, \ie an operation
  \begin{alignat*}{1}
    & \mathsf{ext} : (x : \overline{X})\ (\alpha : \Cof)\ (y : [\alpha] \to \overline{X}) \to \{ \overline{X} \mid \alpha \hra y \}.
  \end{alignat*}

  Then $\overline{X}$ is fibrant and $i$ is surjective (up to paths), \ie $\overline{X}$ is a propositional truncation of $X$.
\end{thm}
\begin{proof}
  We first prove the fibrancy, \ie we construct an element of $\HasWCom_\MI(\overline{X})$.
  By~\autoref{lem:int_tiny}, it suffices to construct an element of $\HasWCom(\lambda \_ \mapsto \overline{X})$.
  This weak composition structure is defined as follows:
  \begin{alignat*}{1}
    & \wcom_{\overline{X}}^{r \to s}(t,b) \triangleq \mathsf{ext}(b, [\alpha \mapsto t(s)]), \\
    & \overline{\wcom}_{\overline{X}}^{r}(t,b) \triangleq \mathsf{ext}(b, [\alpha \mapsto t(s),\ (i=0) \mapsto \wcom_{\overline{X}}^{r \to s}(t,b),\ (i=1) \mapsto b]).
  \end{alignat*}

  Alternatively, we could have used~\autoref{lem:fibrancy_from_reflgraph},
  with $A = 1$, $B = \overline{X}$, $E_B(-) = 1$ and $R_B(-) = 1$.

  Since $X$ is fibrant, we can also construct its propositional truncation $\trunc{X}$ as usual.
  The universal properties of $\trunc{X}$ and $\overline{X}$ provide a logical equivalence $\trunc{X} \lra \overline{X}$, implying that $\overline{X}$ is a propositional truncation of $X$.
\end{proof}

\subsection{Extension structures are finitary}\label{ssec:ext_finitary}

\DeclareRobustCommand{\yo}{\mathbf{y}}

Let $X$ be a cubical set.
Then the data of a global extension structure on $X$ unfolds externally to the following components:
\begin{alignat*}{1}
  & \ext : (I : \square) (\alpha : \Cof(I)) (x : [\alpha] \Rightarrow X) \\
  & \quad \to \{ X(I) \mid \alpha(\id_I) \hra x(\id_I) \}, \\
  & {-} : (f : \square(J,I)) (\alpha : \Cof(I)) (x : [\alpha] \Rightarrow X) \\
  & \quad \to \ext(I,\alpha,x)[F] = \ext(J,\alpha[F],x \circ \yo(F)),
\end{alignat*}
where $\yo$ is the Yoneda embedding and $[\alpha]$ is the subobject of $\yo(I)$ determined by the cofibration $\alpha$ (a sieve on $I$).

The set of natural transformations $x : [\alpha] \Rightarrow X$
can be seen as a limit over the category of elements $\int_\square [\alpha]$.
This category is infinite, so a naive unfolding of the notion of extension structure is infinitary.

However we can use the specific definition of cofibrations in cartesian cubical sets
to replace these limits by finite limits.
\begin{lem}
  For every $I \in \square$ and cofibration $\alpha : \Cof(I)$,
  we can write the set $([\alpha] \Rightarrow X)$ as a finite limit whose shape only depends on $\alpha$,
  naturally in $X$.

  (Equivalently, there is an initial functor $\CC \to \int_\square [\alpha]$ where $\CC$ is a finite category.)
\end{lem}
\begin{proof}
  We prove the result by induction on decompositions of $\alpha$ using interval equality, binary disjunctions and binary conjunctions.
  The result then follows from the following natural isomorphisms:
  \begin{alignat*}{1}
    & ([(0 = 1)] \Rightarrow X) \cong \{ \star \}, \\
    & ([(i = i)] \Rightarrow X) \cong X(I), \\
    & ([(i = j)] \Rightarrow X) \cong X(I \backslash \{i\}),
      \tag*{(When $i \in I$, $j \in I+\{0,1\}$ and $i \neq j$)}
    \\
    & ([\alpha \wedge \beta] \Rightarrow X) \cong ([\alpha] \Rightarrow ([\beta] \Rightarrow X)), \\
    & ([\alpha \vee \beta] \Rightarrow X) \cong ([\alpha] \Rightarrow X) \times_{([\alpha \wedge \beta] \Rightarrow X)} ([\beta] \Rightarrow X).
      \tag*{\qedhere}
  \end{alignat*}
\end{proof}

\begin{cor}
  The notion of extension structure over a cubical set is finitary:
  there is an essentially algebraic theory $\Th_{\CcSet_\ext}$ extending the essentially algebraic theory of cubical sets.
  \qed
\end{cor}

We care more generally about extension structures over global cubical families,
but the argument extends directly to that case.
\begin{cor}
  The notion of extension structure over a cubical family is finitary:
  there is an essentially algebraic theory $\Th_{\mathbf{cFam}_\ext}$ extending the essentially algebraic theory of cubical families.
  \qed
\end{cor}

\begin{cor}\label{cor:finitary_add_ext}
  For every finitary essentially algebraic theory $\Th$ and morphism $F : \Th_{\mathbf{cFam}} \to \Th$,
  there is a finitary essentially algebraic theory $\Th_{\ext_F}$ whose algebras are algebras
  $\CM$ of $\Th$ along with an extension structure on the cubical family $F^\ast(\CM)$.
\end{cor}
\begin{proof}
  The essentially algebraic theory $\Th_{\ext_F}$ is the pushout of $\Th_{\mathbf{cFam}_\ext} \la \Th_{\mathbf{cFam}} \ra \Th$.
\end{proof}


\section{Strict Rezk completions of categories}\label{sec:completion_categories}

In this section, we specify and construct the strict Rezk completions of categories.
Some of the statements of this section may have trivial assumptions,
that is because the theory of categories is a bit too simple:
the category of category has a model structure and every category is cofibrant,
while in general we may want to consider left semi-model structures.
We try to keep the statements and proofs as close to the general case as possible.

\subsection{Categories and their homotopy theory}

We start by giving general definitions that can be interpreted either externally or internally to $\CcSet$.

\begin{defi}
  A \defemph{category} $\CC$ consists of:
  \begin{alignat*}{1}
    & \Ob_\CC : \SSet, \\
    & \Hom_\CC : X \to X \to \SSet, \\
    & \EqHom_\CC : \forall x\ y \to \Hom_\CC(x,y) \to \Hom_\CC(x,y) \to \SSet, \\
    & \id : \forall x \to \Hom_\CC(x,x), \\
    & \_ \circ \_ : \forall x\ y\ z \to \Hom_\CC(y,z) \to \Hom_\CC(x,y) \to \Hom_\CC(x,z), \\
    & \mathsf{idl} : \EqHom_\CC(\id \circ f, f), \\
    & \mathsf{idr} : \EqHom_\CC(\id \circ f, f), \\
    & \mathsf{assoc} : \EqHom_\CC(f \circ (g \circ h), f \circ (g \circ h)), \\
    & \refl : \forall f \to \EqHom_\CC(f,f), \\
    & (p,q : \EqHom_\CC(f,g)) \to (p = q), \\
    & \EqHom_\CC(f,g) \to (f = g).
      \tag*{\qedhere}
  \end{alignat*}
\end{defi}

This presents categories as the algebras of a generalized algebraic theory.
Together, the last rules imply that $\EqHom_\CC(f,g)$ is a proposition equivalent to the equality $(f = g)$.
We could omit the sort $\EqHom$ without changing the definition of category, but that would give a ``wrong'' generalized algebraic theory of categories, \ie one that would not be compatible with the homotopy theory of categories.
The inclusion of $\EqHom$ corresponds to the inclusion of $\{ \bullet \rightrightarrows \bullet \} \to \{ \bullet \rightarrow \bullet \}$ as a generating cofibration in $\CCat$.
It can also be seen as a truncated notion of $2$-cell.
Many definitions need to include conditions for all three sorts, \eg weak equivalences of categories are functors that are essentially surjective on objects, on morphisms (full) and on morphism equalities (faithful).

We may write $x \in \CC$ instead of $x : \Ob_\CC$ and $f \in \CC(x,y)$ instead of $f : \Hom_\CC(x,y)$.

If $\CC$ is a category, we write
\begin{alignat*}{1}
  & \Iso_{\CC}(x,y) \triangleq (f : \Hom_\CC(x,y)) \times (f^{-1} : \Hom_\CC(y,x)) \\
  & \quad \times (f^\eta : \EqHom_\CC(f \circ f^{-1}, \id)) \times (f^{\varepsilon} : \EqHom_\CC(f^{-1} \circ f, \id)).
\end{alignat*}
for the set of isomorphisms between objects $x$ and $y$.

We now recall the main components of the homotopy theory of categories.
\begin{defi}
  A functor $F : \CC \to \CD$ between categories is a \defemph{split weak equivalence} if the following lifting conditions are satisfied:
  \begin{itemize}
  \item For every $x \in \CD$, there is some $x_0 \in X$ and some $p : \Iso_\CD(F(x_0), x)$.
  \item For every $f \in \CD(F(x),F(y))$, there is some $f_0 \in \CC(x,y)$ and some $p : \EqHom_\CD(F(f_0),f)$.
  \item For every $p : \EqHom_\CD(F(f),F(g))$, there is some $p_0 : \EqHom_\CC(f,g)$.
  \end{itemize}
  
  In other words, the split weak equivalences are the functors that are split essentially surjective, full and faithful.
\end{defi}

\begin{prop}
  Split weak equivalences satisfy $2$-out-of-$3$ and are closed under retracts.
  \qed
\end{prop}

\begin{defi}
  A functor $F : \CC \to \CD$ between categories is a \defemph{split trivial fibration} if its actions on objects, morphisms, and morphism equalities are all split surjections:
  \begin{itemize}
  \item For every $x \in \CD$, there is $x_0 \in \CC$ such that $F(x_0) = x$.
  \item For every $f \in \CD(F(x),F(y))$, there is $f_0 \in \CC(x,y)$ such that $F(f_0) = f$.
  \item For every $p : \EqHom_\CD(F(f),F(g))$, there is some $p_0 : \EqHom_\CC(f,g)$ such that $F(p_0) = p$.
  \end{itemize}

  A functor $I : \CA \to \CB$ is an \defemph{algebraic cofibration} if it is equipped with left liftings against all split trivial fibrations.
\end{defi}

\begin{defi}\label{def:cat_fibration}
  A functor $F : \CC \to \CD$ between categories is a \defemph{split fibration} if it satisfies the following lifting condition:
  \begin{itemize}
  \item For every $x \in \CC$ and isomorphism $f : \Iso_\CD(F(x), y)$,
    there is an isomorphism $f_0 : \Iso_\CC(x, y_0)$ such that $F(y_0) = y$ and $F(f_0) = f$.
  \end{itemize}

  A functor $I : \CA \to \CB$ is an \defemph{algebraic trivial cofibration}
  if it is equipped with left liftings against all split fibrations.
\end{defi}

\begin{construction}
  For any category $\CC$, we construct a category $\Path_\CC$, called the \defemph{path-category} of $\CC$, along with projection functors $\pi_1,\pi_2 : \Path_\CC \to \CC$ such that $\pi_1$ and $\pi_2$ are split trivial fibrations and $\angles{\pi_1,\pi_2} : \Path_\CC \to \CC \times \CC$ is a split fibration.
  \begin{itemize}
  \item An object of $\Path_\CC$ is a triple $(x_1,x_2,x_e)$ where $x_e : \Iso_\CC(x_1, x_2)$ is an isomorphism in $\CC$.
  \item A morphism from $(x_1,x_2,x_e)$ to $(y_1,y_2,y_e)$ is a pair $(f_1,f_2)$ where $f_1 : x_1 \to y_1$, $f_2 : x_2 \to y_2$ such that $y_e \circ f_1 = f_2 \circ x_e$.
    \qedhere
  \end{itemize}
\end{construction}

The \defemph{loop-category} $\Loop_\CC$ of a category $\CC$ is the pullback
\[ \begin{tikzcd}
    \Loop_\CC
    \ar[rd, phantom, very near start, "\lrcorner"]
    \ar[r]
    \ar[d, "\pi", two heads]
    &
    \Path_\CC
    \ar[d, "{\angles{\pi_1,\pi_2}}", two heads]
    \\
    \CC
    \ar[r, "\angles{\id,\id}"]
    &
    \CC \times \CC
  \end{tikzcd} \]

\begin{construction}
  We construct a category $\ReflLoop_\CC$, called the \defemph{reflexive-loop-category} of $\CC$, as a displayed category $\ReflLoop_\CC \to \Loop_\CC$, such that $\pi_e : \ReflLoop_\CC \to \Loop_\CC$ is a split fibration and the composition $\pi : \ReflLoop_\CC \to \Loop_\CC \to \CC$ is a split trivial fibration.
  \begin{itemize}
  \item An object of $\ReflLoop_\CC$ displayed over $x_e : \Iso_\CC(x, x)$ is a proof that $x_e = \id$.
  \item There is a unique displayed morphism over every morphism of $\Loop_\CC$.
    \qedhere
  \end{itemize}
\end{construction}

\begin{rem}
  The diagram
  \[ \begin{tikzcd}
      \ReflLoop_\CC
      \ar[r, "\pi_e"]
      & \Path_\CC
      \ar[r, "\pi_1", shift left]
      \ar[r, "\pi_2"', shift right]
      & \CC
    \end{tikzcd} \]
  is a pseudo-reflexive graph object in $\CCat$.
  
  The projection $\pi : \ReflLoop_\CC \to \CC$ is actually an isomorphism.
  In other words, we actually have a reflexive graph object in $\CCat$.
  We try not to rely on this fact, as it won't hold for models of HoTT.
\end{rem}

\begin{prop}
  The constructions of $\Path_\CC$ and $\ReflLoop_\CC$ (and their projection maps) are functorial in $\CC$.
\end{prop}
\begin{proof}
  This follows from the fact that all components of $\Path_\CC$ and $\ReflLoop_\CC$
  are expressed in the language of categories (\eg as finite limits of components of $\CC$).
  More precisely, the functors $\Path_\CC$ and $\ReflLoop_\CC$ are induced by
  morphisms $P,{RL} : \Th_{\CCat} \to \Th_{\CCat}$ of essentially algebraic theories,
  where $\Th_{\CCat}$ is the theory of categories.
\end{proof}

\begin{prop}\label{prop:cat_tcof_is_weq}
  Let $F : \CC \to \CD$ be a functor.
  If $\pi : \ReflLoop_\CC \to \CC$ admits a section $r$ and $F$ is an algebraic trivial cofibration, then $F$ is a split weak equivalence.
\end{prop}
\begin{proof}
  Write $\Path_\CD[F \times\id]$ for the pullback of $\Path_\CD$ over $F \times\id : \CC \times \CD \to \CD \times \CD$.

  Observe that there is a composite map
  \[ r' : \CC \xrightarrow{r} \ReflLoop_\CC \xrightarrow{\pi_e} \Path_\CC \to \Path_\CD[F \times \id] \]
  such that $\pi_1 \circ r' = \id$ and $\pi_2 \circ r' = F$.

  Since $\pi_1 : \Path_\CD[F \times \id] \to \CC$ is a pullback of $\pi_1 : \Path_\CD \to \CD$, it is a split trivial fibration.
  By $2$-out-of-$3$, the map $r'$ is a split weak equivalence.
  
  The map $\pi_2$, as the composition of split fibrations $\Path_\CD[F \times \id] \to \CC \times \CD$ and $\CC \times \CD \to \CD$, is a split fibration.
  Since $F = \pi_2 \circ r'$ has the left lifting property against $\pi_2$, the retract argument says that the map $F$ is a retract of $r'$.
  
  Since split weak equivalences are closed under retracts, $F$ is a split weak equivalence.
\end{proof}

\subsection{Complete categories}

We now work internally to $\CcSet$.

\begin{defi}
  We say that a category $\CC$ has \defemph{fibrant components} if $\Ob_\CC$, $\Hom_\CC(-)$ and $\EqHom_\CC(-)$ are fibrant.
\end{defi}
Note that other sets such as $\Iso_\CC(x, y)$ are also fibrant when $\CC$ has fibrant components.

\begin{defi}\label{def:cat_complete}
  A category $\CC$ with fibrant components is \defemph{complete} if:
  \begin{itemize}
  \item For every $x \in \CC$, the set $(y \in \CC) \times \Iso_\CC(x, y)$ is contractible.
  \item For every $f \in \CC(x,y)$, the set $(g \in \CC(x,y)) \times \EqHom_\CC(f,g)$ is contractible.
  \item For every $p : \EqHom_\CC(f,g)$, the set $\EqHom_\CC(f,g)$ is contractible.
  \end{itemize}

  The first condition says that $\CC$ is univalent~\parencite{UnivalentCategories}.
  The second and third condition always hold due to the isomorphism $\EqHom_\CC(f,g) \cong (f=g)$, but they morally say that $\Hom_\CC$ is a family of h-sets and that $\EqHom_\CC(-)$ is a family of h-propositions.
\end{defi}

\begin{defi}
  A \defemph{strict Rezk completion} of a global category $\CC$ with fibrant components is a global complete category $\overline{\CC}$ along with a global functor $i : \CC \to \overline{\CC}$ such that the external functor $1^\ast_\square(i) : 1^\ast_\square(\CC) \to 1^\ast_\square(\overline{\CC})$ is a weak equivalence.
\end{defi}

There is a bijective correspondence between
global internal categories
and external cubical categories.
Cubical categories are themselves the algebras of a essentially algebraic theory.
When working externally, we write $\CcCat$ for the category of cubical categories.

\subsection{$\Cof$-fibrant categories}

We still work internally to $\CcSet$.

We now give the candidate definition of the strict Rezk completion.
The strict Rezk completion $\overline{\CC}$ of a category $\CC$ should be defined as the free extension of $\CC$ by some additional structure.
The first candidate would be to freely add completeness as defined in~\autoref{def:cat_complete},
but this is poorly behaved in the absence of fibrancy.
We could freely add both completeness and fibrant replacement,
but proving that the externalization of the inclusion
$i : \CC \to \overline{\CC}$ is a weak equivalence would become very hard.

Instead, following~\autoref{thm:trunc_wo_hcom}, we redefine completeness by using trivial fibrancy instead of contractibility.
\begin{defi}
  A \defemph{$\Cof$-fibrancy structure} on a category $\CC$ consists of:
  \begin{itemize}
  \item For every $x \in \CC$, an extension structure $\ext_\Ob(x)$ on $(y \in \CC) \times \Iso_\CC(x, y)$.
  \item For every $f \in \CC(x,y)$, an extension structure $\ext_\Hom(f)$ on $(g \in \CC(x,y)) \times \EqHom_\CC(f,g)$.
  \item For every $p : \EqHom_\CC(f,g)$, an extension structure $\ext_\EqHom(p)$ on $\EqHom_\CC(f,g)$.
    \qedhere
  \end{itemize}
\end{defi}

\begin{rem}\label{rem:cof_fibration}
  Viewing this structure as a kind of fibrancy structure was suggested to the author by Christian Sattler.
  In general, a functor $F : \CC \to \CD$ is a $\Cof$-fibration if for every lifting problem
  (against a generating trivial cofibration of the canonical model structure on $\CCat$),
  the set of diagonal fillers is trivially fibrant.
  More generally,
  we can parametrize these definitions by a notion of cofibration
  (a monomorphism $\Cof \hra \Omega$).
  In the special case when $\Cof = \{\true,\false\}$,
  having an extension structure is the same as having an element,
  and we recover the notion of split fibration from~\autoref{def:cat_fibration}.

  We also note that a category $\CC$ has fibrant components
  if for every lifting problem against a generating cofibration,
  the set of diagonal fillers is fibrant.
\end{rem}

For the theory of categories, only the component $\ext_\Ob$ actually matters, as shown in the following proposition:
\begin{prop}\label{prop:cat_glue_hom_unique}
  Any category can uniquely be equipped with $\ext_\Hom$ and $\ext_\EqHom$.
\end{prop}
\begin{proof}
  This follows from the fact that
  the sets $(g \in \CC(x,y)) \times \EqHom_\CC(f,g)$ and $\EqHom_\CC(f,g)$
  are actually propositions.
\end{proof}

\begin{prop}\label{prop:cat_glue_complete}
  If a category with fibrant components has a $\Cof$-fibrancy structure, then it is complete.
\end{prop}
\begin{proof}
  By~\autoref{prop:fibrant_contractibility_iff_ext}.
\end{proof}

We now switch to the external point of view and consider
the external category $\CcCat_\coffib$ of $\Cof$-fibrant cubical categories.
The $\Cof$-fibrant cubical categories are the algebras of an essentially algebraic theory $\Th_{\CcCat_\coffib}$,
where we use~\autoref{cor:finitary_add_ext} to justify that $\Th_{\CcCat_\coffib}$ is finitary.
There is a forgetful functor $R : \CcCat_\coffib \to \CcCat$.
Because it is induced by a morphism of essentially algebraic theories,
it admits a left adjoint $L : \CcCat \to \CcCat_\coffib$.

Recall from~\autoref{ssec:int_tiny} that we have an endo-adjunction $((-)^\MI \dashv (-)_\MI)$
on the external category $\CcCat$ of cubical categories.
\begin{construction}\label{con:cat_lift_exp_i}
  The functor $(-)^\MI$ lifts to $\CcCat_\coffib$ (along the forgetful functor $R$).
\end{construction}
\begin{proof}[Construction]
  Internally, we have to construct, for every global category $\CC$, a global map
  \[ \ext_{\Ob_{\CC^\MI}} : (\alpha : \Cof) (x \in \CC^\MI) ([\alpha] \to (y \in \CC^\MI) \times \Iso_{\CC^\MI}(x,y))
    \to (y \in \CC^\MI) \times \Iso_{\CC^\MI}(x,y). \]
  We pose
  \[ \ext_{\Ob_{\CC^\MI}}(\alpha,x,e) \triangleq (\lambda i \mapsto \ext_{\Ob_\CC}(\alpha, x(i), e(i))). \]
  Checking functoriality is straightforward.
\end{proof}

\begin{construction}\label{con:cat_lift_sqrt_i}
  The functor $(-)_\MI$ lifts to $\CcCat_\coffib$ (along the forgetful functor $R$).
\end{construction}
\begin{proof}[Construction]
  Internally, we have to construct, for every global category $\CC$, a global map
  \[ \ext_{\Ob_{\CC_\MI}} : (\alpha : \Cof) (x \in \CC_\MI) ([\alpha] \to (y \in \CC_\MI) \times \Iso_{\CC_\MI}(x,y))
    \to (y \in \CC_\MI) \times \Iso_{\CC_\MI}(x,y). \]

  By the universal properties of the components of $\CC_\MI$, it suffices to construct a global map
  \[ E_\CC : (\alpha : \Cof^\MI) (x \in (\CC_\MI)^\MI) (e : (i : \MI) \to [\alpha(i)] \to (y \in \CC_\MI) \times \Iso_{\CC_\MI}(x(i),y))
    \to (y \in \CC) \times \Iso_{\CC}(x,y). \]

  The counit of the adjunction at $\CC$ is a functor $\epsilon_\CC : (\CC_\MI)^\MI \to \CC$.
  
  We pose:
  \[ E_\CC(\alpha,x,e) \triangleq \ext_{\Ob_\CC}(\epsilon_\CC(x), [\forall_i\alpha(i) \mapsto \epsilon_\CC(e)]). \]

  We rely on the closure of cofibrations under $\forall_{i:\MI}$.
  Note that the functor evaluation $\epsilon_\CC(e)$ is only valid under $[\forall_i \alpha(i)]$.
  Indeed, we then have $e : (i : \MI) \to (y \in \CC_\MI) \times \Iso_{\CC_\MI}(x(i),y)$.

  We then know that $\varepsilon(\lambda i. \ext_{\Ob_{\CC_\MI}}(\alpha, x(i), e(i))) = E_\CC(\alpha, x, e)$.
  
  Functoriality follows from the naturality of $E_\CC$.
\end{proof}

\begin{prop}
  The endo-adjunction $((-)^\MI \dashv (-)_\MI)$ lifts to an endo-adjunction on the
  external category of $\Cof$-fibrant cubical categories.
\end{prop}
\begin{proof}
  We have already lifted the two functors in~\autoref{con:cat_lift_exp_i} and~\autoref{con:cat_lift_sqrt_i},
  it remains to lift the rest of the adjunction.

  Take global categories $\CC, \CD$ with $\Cof$-fibrancy structures,
  and functors $F : \CC^\MI \to \CD$ and $G : \CC \to \CD_\MI$
  that are transposes of each others.

  We have to show that $F$ preserves $\Cof$-fibrancy if and only if $G$ preserves $\Cof$-fibrancy.
  Because preservation of $\Cof$-fibrancy is propositional data,
  naturality will hold automatically.

  The functor $F$ preserving $\Cof$-fibrancy means that there is a global proof
  \begin{alignat*}{1}
    & P_1 : (\alpha : \Cof) (x \in \CC^\MI) (e : (i : \MI) [\alpha] \to (y \in \CC) \times \Iso_\CC(x(i),y)) \\
    & \quad \to F(\lambda i. \ext_{\Ob_\CC}(x(i), [\alpha \mapsto e(i)])) = \ext_{\Ob_\CD}(F(x), [\alpha \mapsto F(e)]).
  \end{alignat*}
  This holds if and only if there is a global proof
  \begin{alignat*}{1}
    & P_2 : (\alpha : \Cof^\MI) (x \in \CC^\MI) (e : (i : \MI) [\alpha(i)] \to (y \in \CC) \times \Iso_\CC(x(i),y)) \\
    & \quad \to F(\lambda i. \ext_{\Ob_\CC}(x(i), [\alpha(i) \mapsto e(i)])) = \ext_{\Ob_\CD}(F(x), [\forall_i \alpha(i) \mapsto F(e)]).
  \end{alignat*}
  Indeed, in the forward direction,
  we can pose
  \[ P_2(\alpha,x,e) \triangleq P_1(\forall_i \alpha(i), x, (\lambda i. e(i))), \]
  where $e(i)$, as a partial element defined under $\alpha(i)$,
  is also defined under $\forall_i \alpha$.
  In the reverse direction,
  we can pose
  \[ P_1(\alpha,x,e) \triangleq P_2((\lambda i. \alpha), x, (\lambda i. \ext_{\Ob_\CC}(x(i), [\alpha \mapsto e(i)]))),
  \]
  since $(\forall_i \alpha) \lra \alpha$.

  Now using the fact that $F = \epsilon_\CD \circ G^\MI$,
  we obtain that $F$ preserves $\Cof$-fibrancy if and only if there is a global proof of:
  \begin{alignat*}{1}
  & (\alpha : \Cof^\MI) (x \in \CC^\MI) (e : (i : \MI) [\alpha(i)] \to (y \in \CC) \times \Iso_\CC(x(i),y)) \\
  & \quad \to \epsilon_\CD(\lambda i. G(\ext_{\Ob_\CC}(x(i), [\alpha(i) \mapsto e(i)])))
    = \ext_{\Ob_\CD}(\epsilon_\CD(\lambda i. G(x(i))), [\forall_i \alpha(i) \mapsto \epsilon_\CD(\lambda i. G(e(i)))]).
  \end{alignat*}

  On the left hand side we have the transpose of $G \circ \ext_{\Ob_\CC}$.
  On the right hand side we recognize the transpose of $\ext_{\Ob_{\CD_\MI}}$ from~\autoref{con:cat_lift_sqrt_i}.
  Therefore, by~\autoref{lem:int_tiny}, the above proposition holds if and only there is a global proof of:
  \begin{alignat*}{1}
    & (\alpha : \Cof) (x \in \CC) (e : [\alpha] \to (y \in \CC) \times \Iso_\CC(x,y)) \\
    & \quad \to G(\ext_{\Ob_\CC}(\alpha, x, e)) = \ext_{\Ob_{\CD_\MI}}(G(x), [\alpha \mapsto G(e)]),
  \end{alignat*}
  \ie if $G$ preserves $\Cof$-fibrancy, as needed.
\end{proof}

\begin{prop}\label{prop:exp_i_preserves_cofib}
  The functor $(-)^\MI : \CcCat \to \CcCat$ preserves algebraic cofibrations.
\end{prop}
\begin{proof}
  Using the adjunction $((-)^\MI \dashv (-)_\MI)$, it suffices to prove that $(-)_\MI$ preserves split fibrations,
  Using the adjunction again, it suffices to prove that $(-)^\MI$ sends the generating cofibrations to algebraic cofibrations.
  But the generating cofibrations (in cubical categories) are functors between cubical categories with discrete components,
  which implies that the generating cofibrations are preserved by $(-)^\MI$, as needed.
\end{proof}

\subsection{Fibrancy of the components of the strict Rezk completion}

The strict Rezk completion of a global category $\CC$ will be the $\Cof$-fibrant replacement $\overline{\CC}$ of $\CC$.
We will need to prove that the components of $\overline{\CC}$ are fibrant.
For this purpose, we will use~\autoref{lem:fibrancy_from_reflgraph}.
The pseudo-reflexive graphs will arise from the components of the pseudo-reflexive graph object
  \[ \begin{tikzcd}
      \ReflLoop_{\overline{\CC}}
      \ar[r, "\pi_e"]
      & \Path_{\overline{\CC}}
      \ar[r, "\pi_1", shift left]
      \ar[r, "\pi_2"', shift right]
      & {\overline{\CC}},
    \end{tikzcd} \]
but we also have to check the existence of weak coercion operations and the homotopicality condition.

We now specify notions of weak coercion structures over lines of objects and morphisms in a category $\CC$.
\begin{defi}
  Let $x : \MI \to \Ob_{\CC}$ be a line of objects of a category $\CC$.
  A \defemph{weak coercion structure} on $x$ consists of a family
  \[ \wcoe_x^{r \to s} : \Iso_\CC(x(r), x(s)) \]
  of isomorphisms along with a family
  \[ \wcoh_x^r : \EqHom_{\CC}(\wcoe^{r \to r}_x, \id_{x(r)})
  \]
  of equalities between morphisms.
\end{defi}

\begin{defi}
  Let $f : (i:\MI) \to \Hom_{\CC}(x(i),y(i))$ be a line of morphisms of a category $\CC$.
 
  Given weak coercion structures for $x$ and $y$, a \defemph{weak coercion structure} on $f$ consists of a family
  \[ \wcoe_f^{r \to s} : \EqHom_{\CC}(\wcoe_y^{r \to s} \circ f(r), f(s) \circ \wcoe_x^{r \to s}) \]
  of morphism equalities.
\end{defi}

\begin{construction}
  Given any category $\CC$, we define a displayed category $\HasWCoe^\CC$ over $\CC^\MI$.
  \begin{itemize}
  \item A displayed object over $x : \MI \to \Ob_{\CC}$ is a weak coercion structure $\wcoe_x$ over $x$.
  \item A displayed morphism over $f : (i : \MI) \to \Hom_{\CC}(x(i),y(i))$ is a weak coercion structure $\wcoe_f$ over $f$.
  \item We now define the displayed identity.
    Take a line $x : \MI \to \Ob_{\CC}$ equipped with a weak coercion structure.
    We equip $\id_{x(-)} : (i:\MI) \to \Hom_{\CC}(x(i),x(i))$ with a weak coercion structure:
    \begin{alignat*}{1}
      & \wcoe^{r \to s}_{\id_{x(-)}} : \EqHom_{\CC}(\wcoe^{r \to s}_x \circ \id, \id \circ \wcoe^{r \to s}_x), \\
      & \wcoe^{r \to s}_{\id_{x(-)}} \triangleq \refl.
    \end{alignat*}
  \item We then define the displayed composition.
    Take lines $f : (i:\MI) \to \Hom_{\CC}(y(i),z(i))$ and $g : (i:\MI) \to \Hom_{\CC}(x(i),y(i))$ equipped with weak coercion structures.
    We equip $(f(-) \circ g(-))  : (i:\MI) \to \Hom_{\CC}(x(i),z(i))$ with a weak coercion structure:
    \begin{alignat*}{1}
      & \wcoe^{r \to s}_{f(-) \circ g(-)} : \EqHom_{\CC}(\wcoe^{r \to s}_z \circ f(r) \circ g(r), f(s) \circ g(s) \circ \wcoe^{r \to s}_x).
    \end{alignat*}
    This equality follows from
    \[ \wcoe^{r \to s}_f : \EqHom_\CC(\wcoe^{r \to s}_z \circ f(r), f(s) \circ \wcoe^{r \to s}_y) \]
    and 
    \[ \wcoe^{r \to s}_g : \EqHom_\CC(\wcoe^{r \to s}_y \circ g(r), g(s) \circ \wcoe^{r \to s}_x). \]
  \item
    The interpretations of $\mathsf{idl}$, $\mathsf{idr}$ and $\mathsf{assoc}$ are trivial, any two weak coercion structures over a same morphism are equal.
    \qedhere
  \end{itemize}
\end{construction}

\begin{rem}\label{rem:cat_haswcoe_limit}
  The definition of $\HasWCoe^\CC$ can also be derived from the definitions of $\Path_{-}$ and $\ReflLoop_{-}$.

  For any $r,s : \MI$, we can consider the pullback
  \[ \begin{tikzcd}
      \Path_\CC[\angles{-_r,-_s}]
      \ar[r]
      \ar[rd, phantom, very near start, "\lrcorner"]
      \ar[d]
      &
      \Path_\CC
      \ar[d]
      \\
      \CC^\MI
      \ar[r, "{\angles{-_r,-_s}}"]
      &
      \CC \times \CC
    \end{tikzcd} \]

  For any $r : \MI$, we can then consider the pullback
    \[ \begin{tikzcd}
      \ReflLoop_\CC[\angles{-_r}]
      \ar[r]
      \ar[rd, phantom, very near start, "\lrcorner"]
      \ar[d]
      &
      \ReflLoop_\CC
      \ar[d]
      &
      \\
      \Path_\CC[\angles{-_r,-_r}]
      \ar[rd, phantom, very near start, "\lrcorner"]
      \ar[r]
      \ar[d]
      &
      \Loop_\CC
      \ar[rd, phantom, very near start, "\lrcorner"]
      \ar[r]
      \ar[d]
      &
      \Path_\CC
      \ar[d]
      \\
      \CC^\MI
      \ar[r, "{\angles{-_r}}"]
      &
      \CC
      \ar[r]
      & \CC \times \CC
    \end{tikzcd} \]

  Now consider the diagram shape consisting of objects
  $\mathsf{base}$,
  $\mathsf{path}(r,s)$ for $r,s : \MI$ and
  $\mathsf{refl}\text{-}\mathsf{loop}(r)$ for $r : \MI$,
  such that $\mathsf{base}$ is terminal and with morphisms $\mathsf{refl}\text{-}\mathsf{loop}(r) \to \mathsf{path}(r,r)$ for $r:\MI$.
  Unfolding the constructions shows that $\HasWCoe^\CC$ is the limit of the diagram
  \begin{alignat*}{1}
    & \mathsf{base} \mapsto \CC^\MI, \\
    & \mathsf{path}(r,s) \mapsto \Path_\CC[\angles{-_r,-_s}], \\
    & \mathsf{refl}\text{-}\mathsf{loop}(r) \mapsto \ReflLoop_\CC[\angles{-_r}],
  \end{alignat*}
  with the evident restriction maps.

  In particular, in any situation in which we have path-objects and reflexive-loop-objects,
  we can use this limit as a definition of $\HasWCoe$,
  and obtain notions of weak coercion structures.
\end{rem}

\begin{prop}
  The construction of $\HasWCoe^\CC$ is functorial in $\CC$.
\end{prop}
\begin{proof}
  This follows from the functoriality of $\Path_{-}$ and $\ReflLoop_{-}$.
\end{proof}

\begin{lem}\label{lem:cat_fib_components_map}
  If a category $\CC$ has fibrant components,
  then for any algebraically cofibrant category $\CA$,
  the set $\CCat(\CA,\CC)$ of functors from $\CA$ to $\CC$ is fibrant.
\end{lem}
\begin{proof}
  By~\autoref{lem:int_tiny}, it suffices to construct an element of $\HasWCom(\lambda i \mapsto \CCat(\CA(i),\CC(i)))$ given a line $\CA$ of algebraically cofibrant categories and $\CC$ of categories with fibrant components.

  We define $\CD(s)$ as the equalizer of
  \[
    \begin{tikzcd}
      \CC(r) \times \prod_{s:\MI} \CC(s) \times \prod_{i:\MI} \CC(i)
      \ar[r, shift left, "f_1"]
      \ar[r, shift right, "f_2"']
      & \CC(r)^{[\alpha]} \times \prod_{i:\MI} \CC(i)^{[\alpha]}
    \end{tikzcd}
  \]
  where $f_1(b,w,\underline{w}) \triangleq (w(r), \underline{w})$
  and $f_2(b,w,\underline{w}) \triangleq (b, \lambda \_ \mapsto b)$.
  (More precisely, using combinators for categorical products,
  $f_1 = \angles{ \angles{ \pi_r \circ \pi_2 }_{-:[\alpha]}, \angles{\pi_3}_{-:[\alpha]} }$
  and $f_2 = \angles{ \angles{ \pi_1 }_{-:[\alpha]}, \angles{ \pi_1 }_{i:\MI,-:[\alpha]} }$.)

  We define $\CE(s)$ as the equalizer of
  \[
    \begin{tikzcd}
      \CC(r) \times \prod_{s:\MI} \CC(s)^{[\alpha]}
      \ar[r, shift left, "g_1"]
      \ar[r, shift right, "g_2"']
      & \CC(r)^{[\alpha]}
    \end{tikzcd}
  \]
  where $g_1(b,t) \triangleq t(r)$ and $g_2(b,t) \triangleq b$.
  (More precisely, using combinators for categorical products,
  $g_1 = \angles{ \pi_r \circ \pi_2 }_{-:[\alpha]}$
  and $g_2 = \angles{ \pi_1 }_{-:[\alpha]}$.)
  
  We have a functor $p : \CD(s) \to \CE(s)$,
  determined by the same formula
  \[ p(b,w,\underline{w}) \triangleq (b, \lambda s \mapsto w(s))
  \] as for~\autoref{lem:wcom_as_surj}.
  Since limits and computed sortwise in categories, the actions of $p$ on objects, morphisms and morphism equalities are given by the map of~\autoref{lem:wcom_as_surj}.
  Thus, since $\CC(s)$ has fibrant components, \autoref{lem:wcom_as_surj} implies that the actions of $p$ on each sort are split surjections, \ie that the map $p : \CD(s) \to \CE(s)$ is a split trivial fibration.

  Now take $t : [\alpha] \to (s : \MI) \to \CCat(\CA(s),\CC(s))$ and $b : \CCat(\CA(r),\CC(r))$ such that $[\alpha] \to t(r) = b$.
  We have a map $\angles{b,t} : \CA(s) \to \CE(s)$.
  Since $\CA(s)$ is algebraically cofibrant, this map factors through $p$.
  This factor can be decomposed into $\wcom_{\CCat(\CA(-),\CC(-))}^{r \to s}(t,b)$ and $\underline{\wcom}_{\CCat(\CA(-),\CC(-))}^{r}(t,b)$, as needed.
\end{proof}

\begin{prop}\label{prop:cat_haswcoe_tfib}
  If a category $\CC$ has fibrant components, then the projection $\HasWCoe^\CC \to \CC^\MI$ is a split trivial fibration.
\end{prop}
\begin{proof}
  Note that this amounts to checking the following conditions:
  \begin{itemize}
  \item For every line $x : \MI \to \Ob_\CC$, there is a weak coercion structure over $x$.
  \item For every line $f : (i : \MI) \to \Hom_\CC(x(i),y(i))$, and given weak coercion structures over $x$ and $y$, there is a weak coercion structure over $f$.
  \item There is also a condition for equalities between morphisms, but it is trivial since morphism equalities are trivial in $\HasWCoe^\CC$.
  \end{itemize}

  We prove it more abstractly using the definition of $\HasWCoe$ as a limit.
  Let $I : \CA \to \CB$ be a (generating) trivial cofibration and take a lifting problem
  \[ \begin{tikzcd}
      \CA
      \ar[r, "F"]
      \ar[d, "I"]
      &
      \HasWCoe^\CC
      \ar[d]
      \\
      \CB
      \ar[r, "G"]
      &
      \CC^\MI \rlap{\ .}
    \end{tikzcd} \]

  We construct a diagonal lift $\CB \to \HasWCoe^\CC$ using the universal property of $\HasWCoe^\CC$ as a limit.
  This means that we have to construct diagonal lifts $K_{r,s} : \CB \to \Path_\CC[\angles{-_r,-_s}]$ for any $r,s:\MI$ and $L_r : \CB \to \ReflLoop_\CC[\angles{-_r}]$ for any $r:\MI$ such that $\pi_e \circ L_r = K_{r,r}$.
  Since $\CC$ has fibrant components, $\Path_\CC[\angles{-_r,-_s}]$ also has fibrant components.
  Thus, by~\autoref{lem:cat_fib_components_map}, it suffices to define $K_{r,s}$ under the diagonal cofibration $(r = s)$.
  Since $\pi : \ReflLoop_\CC \to \CC$ is a split trivial fibration, so is its pullback $\pi : \ReflLoop_\CC[\angles{-_r}] \to \CC^\MI$.
  This provides the lifts $L_r$, and we can pose $K_{r,r} = \pi_e \circ L_r$.
\end{proof}

Now assume that $\CC$ be a global algebraically cofibrant category with fibrant components.
\begin{construction}\label{constr:strict_rezk_cat}
  We write $\overline{\CC}$ for the $\Cof$-fibrant replacement of $\CC$, \ie the category freely generated by a functor $i : \CC \to \overline{\CC}$ and a $\Cof$-fibrancy structure.
\end{construction}
We will have to prove two things: the fibrancy of the components of $\overline{\CC}$, and the fact that $1^\ast_\square(i) : 1^\ast_\square(\CC) \to 1^\ast_\square(\overline{\CC})$ is an external split weak equivalence of categories.

\begin{lem}\label{lem:cat_coffib_commutes_expi}
  The map $i^\MI : \CC^\MI \to \overline{\CC}^\MI$ exhibits $\overline{\CC}^\MI$ as a $\Cof$-fibrant replacement of $\CC^\MI$.
\end{lem}
\begin{proof}
  We work externally.
  There is an adjunction $(L \dashv R)$ between the categories $\CcCat$ and $\CcCat_\coffib$.
  We also have an endo-adjunction $((-)^\MI \dashv (-)_\MI)$ on $\CcCat$ which lifts (along $R$) to $\CcCat_\coffib$.
  We know that $R$ commutes strictly with all components of the adjunctions $((-)^\MI \dashv (-)_\MI)$;
  in other words, $R$ determines a morphism
  from the adjunction $((-)^\MI \dashv (-)_\MI)$ on $\CcCat_\coffib$ to
  the adjunction $((-)^\MI \dashv (-)_\MI)$ on $\CcCat$.

  In such a situation, we obtain, for every $\CC : \CcCat$, an adjunction
  between the comma categories $(\CC \downarrow R)$ and $(\CC^\MI \downarrow R)$.
  The left adjoint
  sends $(\CD : \CcCat_\coffib, f : \CC \to R(\CD))$
  to $(\CD^\MI, f^\MI)$,
  relying on the fact that $R(\CD^\MI) = R(\CD)^\MI$.
  The right adjoint sends $(\CD : \CcCat_\coffib, g : \CC^\MI \to R(\CD))$
  to $(\CD_\MI, g_\MI \circ \varepsilon_{\CC}) : \CC \to R(\CD_\MI)$.

  This left adjoint preserves initial object,
  which tells us that $L(\CC)^\MI \cong L(\CC^\MI)$ and $\eta_{\CC}^\MI \cong \eta_{\CC^\MI}$,
  where $\eta$ is the unit of the adjunction $(L \dashv R)$.
\end{proof}

\begin{prop}\label{prop:cat_haswcoe_glue}
  The displayed category $\HasWCoe^{\overline{\CC}}$ can be equipped with a displayed $\Cof$-fibrancy structure.
\end{prop}
\begin{proof} We interpret the operations of a $\Cof$-fibrancy structure.
  By~\autoref{prop:cat_glue_hom_unique}, we only have to interpret $\ext_\Ob$.
  \begin{description}
  \item[Interpretation of $\ext_\Ob$]
    Take a cofibration $\alpha$ and lines $x : \MI \to \overline{\CC}$, $y : (i : \MI) \to [\alpha] \to \Ob_{\overline{\CC}}$ and $e : (i : \MI) \to [\alpha] \to \Iso_{\overline{\CC}}(x(i), y(i))$.
    We need to define displayed object (weak coercion structures) over the lines $G(-) = \Glue_\Ob(x(-),y(-),e(-))$ (of objects) and $g(-) = \glue_\Ob(x(-),y(-),e(-))$ (of isomorphisms).
    A weak coercion structure for the isomorphism $g$ consists of weak coercion structures for both morphisms $g$ and $g^{-1}$.
    The coercion structures should also coincide with $\wcoe_y$ and $\wcoe_e$ under the cofibration $\alpha$.

    We pose
    \begin{alignat*}{1}
      & \wcoe^{r \to s}_{G} : \Iso_{\overline{\CC}}(G(r), G(s)), \\
      & \wcoe^{r \to s}_{G} \triangleq {g(s)} \circ \wcoe^{r \to s}_x \circ {g(r)}^{-1}, \\
      & \wcoh^{r}_{G} : \EqHom_{\overline{\CC}}(\wcoe^{r \to r}_{G}, \id), \\
      & \wcoe^{r \to s}_{g} : \EqHom_{\overline{\CC}}(\wcoe^{r \to s}_{G} \circ g(r), g(s) \circ \wcoe_x^{r \to s}), \\
      & \wcoe^{r \to s}_{g^{-1}} : \EqHom_{\overline{\CC}}(\wcoe^{r \to s}_{x} \circ g(r)^{-1}, g(s)^{-1} \circ \wcoe_G^{r \to s}).
    \end{alignat*}
    The equality $\wcoh^{r}_G$ follows from $\wcoh^r_x$, which says that $\wcoe^{r \to r}_x = \id_x$.
    The equalities $\wcoe^{r \to s}_g$ and $\wcoe^{r \to s}_{g^{-1}}$ follow from the definition of $\wcoe^{r \to s}_G$ and the categorical laws.

    We then have to check that under $\alpha$, these restrict to $\wcoe_y$, $\wcoh_y$, $\wcoe_e$ and $\wcoe_{e^{-1}}$.
    We already know that $G$ and $g$ restrict to $y$ and $e$.
    Only the case of $\wcoe_y$ is non-trivial: it follows from the equality $\wcoe_e$ between $\wcoe^{r \to s}_y \circ e(r)$ and $e(s) \circ \wcoe^{r \to s}_x$.
    \qedhere
  \end{description}
\end{proof}

\begin{prop}\label{prop:cat_sect_haswcoe}
  The displayed category $\HasWCoe^{\overline{\CC}} \to \overline{\CC}^\MI$ admits a global section.
\end{prop}
\begin{proof}
  By~\autoref{prop:cat_haswcoe_tfib}, $\HasWCoe^{\CC} \to \CC^\MI$ is a split trivial fibration.
  By~\autoref{prop:exp_i_preserves_cofib}, the functor $(-)^\MI$ preserves algebraic cofibrations,
  so $\CC^\MI$ is algebraically cofibrant and $\HasWCoe^{\CC} \to \CC^\MI$ admits a section.
  By composing this section with $\HasWCoe^\CC \to \HasWCoe^{\overline{\CC}}$, we obtain a map $\CC^\MI \to \HasWCoe^{\overline{\CC}}$ displayed over $i^\MI : \CC^\MI \to \overline{\CC}^\MI$.

  By combining this with~\autoref{prop:cat_haswcoe_glue}, we can use the universal property of $\overline{\CC}^\MI$ from~\autoref{lem:cat_coffib_commutes_expi} to obtain a section of $\HasWCoe^{\overline{\CC}} \to \overline{\CC}^\MI$.
\end{proof}

\begin{prop}\label{prop:strict_rezk_cat_fibrant}
  The category $\overline{\CC}$ has fibrant components.
\end{prop}
\begin{proof}
  We need to define weak composition structures for the objects, morphisms and morphism equalities of $\overline{\CC}$.
  \begin{description}
  \item[Weak composition for $\Ob_{\overline{\CC}}$]
    We use~\autoref{lem:fibrancy_from_reflgraph} for $A = \top$ and $B = \lambda \_ \to \Ob_{\overline{\CC}}$.
    The families $E_B$ and $R_B$ are given by the corresponding components of $\Path_{\overline{\CC}}$ and $\ReflLoop_{\overline{\CC}}$, namely
    \begin{alignat*}{1}
      & E_B(x,y) = \Iso_{\overline{\CC}}(x,y), \\
      & R_B(x,f) = \EqHom_{\overline{\CC}}(f,\id).
    \end{alignat*}
    By~\autoref{prop:cat_sect_haswcoe}, we have the required operations $\wcoe$ and $\wcoh$.

    It remains to check homotopicality, \ie to construct extension structures
    \begin{alignat*}{1}
      & \forall x \to \HasExt((y : \Ob_{\overline{\CC}}) \times \Iso_{\overline{\CC}}(x,y)), \\
      & \forall x \to \HasExt((f : \Iso_{\overline{\CC}}(x,x)) \times \EqHom_{\overline{\CC}}(f,\id)).
    \end{alignat*}

    They both arise from $\overline{\CC}$ being $\Cof$-fibrant.
    
  \item[Weak composition for $\Hom_{\overline{\CC}}$]
    We use~\autoref{lem:fibrancy_from_reflgraph} for $A = \Ob_{\overline{\CC}} \times \Ob_{\overline{\CC}}$ and $B(x,y) = \Hom_{\overline{\CC}}(x,y)$.
    The families $E_A$, $R_A$, $E_B$ and $R_B$ are given by the corresponding (limits of) components of $\Path_{\overline{\CC}}$ and $\ReflLoop_{\overline{\CC}}$, \eg
    $E_A((x_1,y_1),(x_2,y_2)) = \Iso_{\overline{\CC}}(x_1,x_2) \times \Iso_{\overline{\CC}}(y_1,y_2)$.
    By~\autoref{prop:cat_sect_haswcoe}, we have the required operations $\wcoe$ and $\wcoh$.
    
    It remains to construct extension structures
    \begin{alignat*}{1}
      & \forall (x_e : \Iso_{\overline{\CC}}(x_1,x_2)) (y_e : \Iso_{\overline{\CC}}(y_1,y_2)) (f_1 \in \Hom_{\overline{\CC}}(x_1,y_1)) \\
      & \quad \to \HasExt((f_2 : \Hom_{\overline{\CC}}(x_2,y_2)) \times \EqHom_{\overline{\CC}}(f_2 \circ x_e, y_e \circ f_1)), \\
      & \forall (x_e : \Iso_{\overline{\CC}}(x,x)) (x_r : \EqHom_{\overline{\CC}}(x_e,\id)) (y_e : \Iso_{\overline{\CC}}(x,x)) (y_r : \EqHom_{\overline{\CC}}(y_e,\id)) \\
      & \phantom{\forall} (f \in \Hom_{\overline{\CC}}(x,y)) \\
      & \quad \to \HasExt((f : \Iso_{\overline{\CC}}(x,x)) \times \EqHom_{\overline{\CC}}(f,\id)).
    \end{alignat*}

    The set $(f_2 : \Hom_{\overline{\CC}}(x_2,y_2)) \times \EqHom_{\overline{\CC}}(f_2 \circ x_e, y_e \circ f_1)$ has a unique element $(y_e \circ f_1 \circ x_e^{-1}, \refl)$.
    The set $(f : \Iso_{\overline{\CC}}(x,x)) \times \EqHom_{\overline{\CC}}(f,\id)$ has a unique element $(\id, \refl)$.
    This provides the needed extension structures.
    
  \item[Weak composition for $\EqHom_{\overline{\CC}}$]
    We could proceed as above, but this follows more directly from~\autoref{prop:fibrant_contractibility_iff_ext}, since $\EqHom_{\overline{\CC}}(f,g)$ is trivially fibrant.
    \qedhere
  \end{description}
\end{proof}

\begin{rem}
  In the proof of~\autoref{prop:strict_rezk_cat_fibrant} we had to prove the trivial fibrancy of some sets in order to ensure homotopicality.
  An alternative would be to add these as additional extension structures in the definition of $\Cof$-fibrancy structure.
  This would simplify the proof of~\autoref{prop:strict_rezk_cat_fibrant},
  at the price of additional cases in the proof of~\autoref{prop:cat_haswcoe_glue}.

  The cleanest approach is probably not to add these additional extension structures,
  but show that they arise from the trivial fibrancy of spaces of solutions to some lifting problem against cofibrations in $\CCat$.
  This would require studying in more details the notion of $\Cof$-fibration,
  its closure properties and how it interacts with the other classes of maps.
\end{rem}

We have now proven that $\overline{\CC}$ has fibrant components.
It remains to prove that $1_\square^\ast(i) : 1_\square^\ast(\CC) \to 1_\square^\ast(\overline{\CC})$ is an external split weak equivalence.

\begin{lem}\label{lem:ext_cof_fibrant_is_fibrant}
  The external category $1^\ast_\square(\overline{\CC})$ has the universal property of the fibrant replacement of $1^\ast_\square(\CC)$.
\end{lem}
\begin{proof}
  This almost follows from the fact that the left adjoint $1^\ast_\square$ preserves colimits, except for the fact that the definition of $\overline{\CC}$ depends on the internal notion of cofibration.
  This can be seen as a crisp induction principle for $\overline{\CC}$, and could alternatively be proven in the spatial type theory of~\textcite{ShulmanCrisp}.
  
  A fibrancy structure on a category $\CX$ is an operation
  \begin{alignat*}{1}
    & \ext_\Ob : (x \in \CX) \to (y \in \CX) \times (f : \Iso_\CX(x,y)).
  \end{alignat*}
  (Operations $\ext_\Hom$ and $\ext_\EqHom$ would be trivial for the same reasons as~\autoref{prop:cat_glue_hom_unique}.)

  We work externally.
  
  Write $\CCat_\fib$ for the category of categories with a fibrancy structure,
  $\CcCat_\fib$ for the category of cubical categories with fibrancy structures,
  and $\CCat_{\Cof\text{-}\fib}$ for the category of cubical categories with a $\Cof$-fibrancy structure.

  There is a functor $\CcCat_{\Cof\text{-}\fib} \to \CcCat_\fib$,
  obtained by specializing the $\Cof$-fibrancy structure to the cofibration $\bot$.
  Moreover, this functor is induced by a morphism of essentially algebraic theories.

  We have the following diagram, where the bottom part is a morphism of adjunctions.
  \[ \begin{tikzcd}[row sep=40pt, column sep=30pt]
      \CcCat_{\coffib}
      \ar[d]
      \\ \CcCat_{\fib}
      \ar[r, bend right=15, "{1^\ast_\square}"']
      \ar[r, phantom, "\top"]
      \ar[d]
      & \CCat_\fib
      \ar[d]
      \ar[l, bend right=15, "{{(1_\square)}_\ast}"']
      \\ \CcCat
      \ar[r, bend right=15, "{1^\ast_\square}"']
      \ar[r, phantom, "\top"]
      & \CCat
      \ar[l, bend right=15, "{{(1_\square)}_\ast}"']
    \end{tikzcd} \]

  Now for any $\CC : \CCat_\fib$,
  we prove that ${(1_\square)}_\ast(\CC)$ admits a unique $\Cof$-fibrancy structure extending its fibrancy structure.
  The component of the $\Cof$-fibrancy structure for $\alpha : \Cof$ consists of a map
  \[ (x : {(1_\square)}_\ast(X)) (y : [\alpha] \to {(1_\square)}_\ast(Y(x))) \to {(1_\square)}_\ast(\{ Y(x) \mid \alpha \hra y \}) \]
  for some sets $X$ and $Y$.
  By properties of the right adjoint ${(1_\square)}_\ast$, we can assume that $\alpha$, $x$ and $y$ are global elements.
  In particular, the cofibration $\alpha$ is either $\top$ or $\bot$, since global cofibrations are decidable.
  When $\alpha = \top$, the component of the fibrancy structure is uniquely determined.
  Thus, a $\Cof$-fibrancy structure on ${(1_\square)}_\ast(\CC)$ is uniquely determined by its component for $\alpha = \bot$.
  In other words, ${(1_\square)}_\ast(\CC)$ has a unique $\Cof$-fibrancy structure extending its fibrancy structure.

  This implies that we have a morphism of adjunctions
  \[ \begin{tikzcd}[row sep=40pt, column sep=30pt]
      \CcCat_{\coffib}
      \ar[r, bend right=15, "{1^\ast_\square}"']
      \ar[r, phantom, "\top"]
      \ar[d, "R"]
      & \CCat_\fib
      \ar[d, "R"]
      \ar[l, bend right=15, "{{(1_\square)}_\ast}"']
      \\ \CcCat
      \ar[r, bend right=15, "{1^\ast_\square}"']
      \ar[r, phantom, "\top"]
      & \CCat
      \ar[l, bend right=15, "{{(1_\square)}_\ast}"']
    \end{tikzcd} \]

  In this situation, we have, for any $\CC : \CcCat$, an adjunction between commas categories
  $(\CC \downarrow R)$ and $(1^\ast_\square(\CC) \downarrow R)$,
  whose left adjoint sends $(\CD : \CcCat_\coffib, f : \CC \to R(\CD))$ to $(1^\ast_\square(\CD), 1^\ast_\square(f))$.
  The preservation of the initial object then says
  that $1^\ast_\square$ sends the $\Cof$-fibrant replacement of $\CC$ to the fibrant replacement of $1^\ast_\square(\CC)$,
  as needed.
\end{proof}

\begin{prop}\label{prop:strict_rezk_completion_cat_weq}
  The functor $1^\ast_\square(i) : 1^\ast_\square(\CC) \to 1^\ast_\square(\overline{\CC})$ is a split weak equivalence.
\end{prop}
\begin{proof}
  We first show that $\pi : \ReflLoop_{1^\ast_\square(\CC)} \to 1^\ast_\square(\CC)$ admits a section.
  Since $\CC$ is algebraically cofibrant and $\pi_e : \ReflLoop_\CC \to \CC$ is a split trivial fibration, we have a section $r$ of $\pi : \ReflLoop_\CC \to \CC$, hence a section $1^\ast_\square(r)$ of $1^\ast_\square(\pi) : 1^\ast_\square(\ReflLoop_\CC) \to 1^\ast_\square(\CC)$.
  Since $1^\ast_\square$ preserves finite limits and the components of $\ReflLoop_\CC$ are finite limits of components of $\CC$, we have $1^\ast_\square(\ReflLoop_\CC) = \ReflLoop_{1^\ast_\square(\CC)}$.
  Thus, $1^\ast_\square(r)$ is a section of $\pi : \ReflLoop_{1^\ast_\square(\CC)} \to 1^\ast_\square(\CC)$.
  
  By~\autoref{lem:ext_cof_fibrant_is_fibrant}, we know that $1^\ast_\square(i) : 1^\ast_\square(\CC) \to 1^\ast_\square(\overline{\CC})$ is an algebraic trivial cofibration.

  We have verified the conditions of~\autoref{prop:cat_tcof_is_weq}.
  Therefore, $1^\ast_\square(i) : 1^\ast_\square(\CC) \to 1^\ast_\square(\overline{\CC})$ is a split weak equivalence.
\end{proof}

\begin{thm}
  Any global cofibrant category with fibrant components admits a strict Rezk completion.
\end{thm}
\begin{proof}
  We use the category $\overline{\CC}$ defined in~\autoref{constr:strict_rezk_cat}.
  By~\autoref{prop:strict_rezk_cat_fibrant} it has fibrant components.
  By~\autoref{prop:cat_glue_complete} it is complete.
  By~\autoref{prop:strict_rezk_completion_cat_weq}, the functor $i : \CC \to \overline{\CC}$ is a split weak equivalence.
\end{proof}


\section{Semantics of HoTT}\label{sec:hott}

In this section we describe the semantics of HoTT, \ie we describe its category of models and some model constructions.

We choose a variant of HoTT in which every type belongs to some universe.
The types are stratified by a hierarchy of $\omega$ universes $(\UU_n)_{n<\omega}$, and types at level $n$ are in bijective correspondence with terms of $\UU_n$.
Using this variant, it suffices to consider terms in many constructions, instead of dealing with terms and types separately.

\subsection{Families}

We first describe HoTT as a second-order theory, \ie using higher-order abstract syntax.

\begin{defi}
  A \defemph{cumulative family} consists of the following components, where $n$ ranges over natural numbers:
  \begin{alignat*}{1}
    & \Ty_n : \SSet, \\
    & \Tm_n : \Ty_n \to \SSet, \\
    & \Lift_n : \Ty_n \to \Ty_{n+1}, \\
    & \lift : \Tm_n(A) \cong \Tm_{n+1}(\Lift_n(A)), \\
    & \UU_n : \Ty_{n+1}, \\
    & \El : \Tm(\UU_n) \cong \Ty_n.
      \tag*{\qedhere}
  \end{alignat*}
\end{defi}

If $\MM$ is a cumulative family, we may write $\MM_n$ instead of $\MM.\Ty_n$.
We also omit $\Tm(-)$ when possible.
For instance, given a type $A : \MM_n$, the set of dependent type may be written $A \to \MM_n$ instead of $\MM.\Tm(A) \to \MM.\Ty_n$.
We similarly omit $\Lift(-)$, $\lift(-)$ and $\El(-)$ when unambiguous.

\begin{defi}
  A \defemph{MLTT-family} is a cumulative family equipped with the structures of $\Pi$-types with function extensionality, $\Sigma$-types, $\Unit$-types, $\Id$-types, boolean types, empty types and $W$-types.
\end{defi}


The following definitions of contractible types and equivalences are used to specify univalent universes.
\begin{alignat*}{1}
  & \isContr(A) \triangleq (x : A) \times ((y : A) \to \Id_A(x,y)), \\
  & \isEquiv(f) \triangleq (b : B) \to \isContr((a : A) \times \Id_B(f(a),b)), \\
  & \Equiv(A,B) \triangleq (f : A \to B) \times \isEquiv(f).
\end{alignat*}

\begin{defi}
  A \defemph{univalence structure} on a MLTT-family consists of operations
  \begin{alignat*}{1}
    & \ua_n : (A : \UU_n) \to \isContr((B : \UU_n) \times \Equiv(A,B)).
      \tag*{\qedhere}
  \end{alignat*}
\end{defi}

\begin{defi}
  A \defemph{HoTT-family} is a MLTT-family equipped with a univalence structure.
\end{defi}

\subsection{Models}

Our notion of models is based on categories with families~\parencite{InternalTypeTheory,CwFsUSD}.

\begin{defi}
  A \defemph{model of HoTT} is a category $\CM$, with a terminal object $1_\CM$, together with a global HoTT-family $(\CM.\Ty,\CM.\Tm,\dots)$ in $\CPsh(\CM)$, such that for every $n < \omega$, the dependent presheaf $\CM.\Tm_n$ is locally representable.
\end{defi}

If $\CM$ is a model, we will write $a \Colon X$ to indicate that $A$ is a global element of a global type $X$ of the presheaf model $\CPsh(\CM)$.
In particular, we may write $A \Colon \CM.\Ty_n$ (or $A \Colon \UU_n$) to indicate that $A$ is a closed type, or $A \Colon \yo(\Gamma) \to \CM.\Ty_n$ to indicate that $A$ is a type over $\Gamma \in \CM$.

We will sometimes need to restrict to democratic models.
\begin{defi}
  A model $\CM$ is \defemph{democratic} if for every object $\Gamma \in \CM$, there is a closed type $K(\Gamma)$ and an isomorphism $1.K(\Gamma) \cong \Gamma$.
\end{defi}
Given a democratic model, we will identify contexts with closed types and omit the operation $K(-)$ and the isomorphism $1.K(\Gamma) \cong \Gamma$.

The category of models of HoTT can be equipped with classes of weak equivalences, fibrations, trivial fibrations, \etc
These correspond to the classes of maps introduced by~\textcite{HomotopyTheoryTTs}.
These classes of maps are \emph{local}, in the sense that their lifting conditions only involve types and terms, not objects and morphisms.
As a consequence, they are only well-behaved when restricted to democratic models.
Because types are in bijective correspondence with terms of universes, we can omit any lifting condition involving types from our definitions.
Only lifting conditions for terms are needed.

\begin{defi}
  A morphism $F : \CM \to \CN$ between models of HoTT is a \defemph{split weak equivalence} if the following weak lifting condition is satisfied:
  \begin{description}
  \item[Weak term lifting] For every type $A : \CM.\Ty(\Gamma)$ and term $a : \CN.\Tm(F(\Gamma),F(A))$, there is a term $a_0 : \CM.\Tm(\Gamma,A)$ and an identification $p : \CN.\Tm(F(\Gamma), \Id_{F(A)}(F(a_0),a))$.
    \qedhere
  \end{description}
\end{defi}

\begin{defi}
  A morphism $F : \CM \to \CN$ between models of HoTT is a \defemph{split trivial fibration} if the following lifting condition is satisfied:
  \begin{description}
  \item[Term lifting] For every type $A : \CM.\Ty(\Gamma)$ and term $a : \CN.\Tm(F(\Gamma),F(A))$, there is a term $a_0 : \CM.\Tm(\Gamma,A)$ such that $F(a_0) = a$.
  \end{description}

  A morphism $I : \CA \to \CB$ is an \defemph{algebraic cofibration} if it is equipped with lifting structures against all split trivial fibrations.
\end{defi}

\begin{defi}
  A morphism $F : \CM \to \CN$ between models of HoTT is a \defemph{split fibration} if the following lifting condition is satisfied:
  \begin{description}
  \item[Identification lifting] For every term $a : \CM.\Tm(\Gamma,A)$ and identification
    \[ p : \CN.\Tm(F(\Gamma), \Id_{F(A)}(F(x),y)), \]
    there is a term $y_0 : \CM.\Tm(\Gamma,A)$ and an identification $p_0 : \CM.\Tm(\Gamma,\Id_A(x,y_0))$ such that $F(y_0) = y$ and $F(p_0) = p$.
  \end{description}

  A morphism $I : \CA \to \CB$ is an \defemph{algebraic trivial cofibration} if it is equipped with lifting structures against all split fibrations.
\end{defi}

\begin{prop}\label{prop:hott_two_out_of_three}
  Split weak equivalences between democratic models satisfy the $2$-out-of-$3$ property. 
\end{prop}
\begin{proof}
  Let $F : \CC \to \CD$ and $G : \CD \to \CE$ be two composable morphisms between democratic models.
  \begin{description}
  \item[{1}]
    Assume that both $F$ and $G$ are weak equivalences.
    
    Take a term $a : \CE.\Tm(G(F(\Gamma)),G(F(A))$.
    Since $G$ is a weak equivalence, there is $a_0 : \CD.\Tm(F(\Gamma),F(A))$ and $p_0 : \CE.\Tm(G(F(\Gamma)),\Id(G(a_0),a))$.
    Since $F$ is a weak equivalence, there is $a_1 : \CC.\Tm(\Gamma,A)$ and $p_1 : \CD.\Tm(F(\Gamma),\Id(F(a_1),a_0))$.
    Then $(G(p_1) \cdot p_0) : \CE.\Tm(G(F(\Gamma)), \Id(G(F(a_1)), a))$ witnesses the fact that $a_1$ is a weak lift of $a$.
    Thus $(G \circ F)$ is a weak equivalence.
  \item[{2}]
    Assume that both $G$ and $(G \circ F)$ are weak equivalences.
    
    Take a term $a : \CD.\Tm(F(\Gamma),F(A))$.
    Since $(G \circ F)$ is a weak equivalence, there is $a_0 : \CC.\Tm(\Gamma,A)$ and $p_0 : \CE.\Tm(G(F(\Gamma)),\Id(G(F(a_0)),G(a)))$.
    Since $G$ is a weak equivalence, there is $p_1 : \CD.\Tm(F(\Gamma), \Id(F(a_0),a)$, exhibiting $a_0$ as a weak lift of $a$.
    Thus $F$ is a weak equivalence.
  \item[{3}]
    Assume that both $F$ and $(G \circ F)$ are weak equivalences.
    
    Take a term $a : \CE.\Tm(G(\Gamma),G(A))$.
    Since $F$ is a weak equivalence and $\CD$ is democratic, there is $\Gamma_0 \in \CC$ and $A_0 : \CC.\Ty(\Gamma_0)$, an equivalence $\alpha$ between $F(\Gamma_0)$ and $\Gamma$ and a dependent equivalence $\beta$ between $F(A_0)$ and $A$ lying over $\alpha$.
    We can transport $a$ over $G(\alpha)$ and $G(\beta)$ to obtain a term $a_0 : \CE.\Tm(G(F(\Gamma_0)),G(F(A_0))$.
    Since $(G \circ F)$ is a weak equivalence, there is a lift $a_1 : \CC.\Tm(\Gamma_0,A_0)$ and $p_1 : \CE.\Tm(G(F(\Gamma_0)), \Id(G(F(a_1)), a_0))$.
    Now define $a_2 : \CD.\Tm(\Gamma,A)$ by transporting $a_1$ over $\alpha$ and $\beta$.
    The transports cancel out in $G(a_2)$, and we obtain an identification $p_2 : \CE.\Tm(G(\Gamma_0), \Id(G(a_2), a))$, exhibiting $a_2$ as a weak lift of $a$.
    Thus $G$ is a weak equivalence.
    \qedhere
  \end{description}
\end{proof}

\begin{prop}\label{prop:hott_weq_retracts}
  Split weak equivalences are closed under retracts. 
\end{prop}
\begin{proof}
  Take a retract diagram
  \[ \begin{tikzcd}
    \CA
    \ar[r, "S_1"]
    \ar[d, "G"]
    & \CB
    \ar[d, "F"]
    \ar[r, "R_1"]
    & \CA
    \ar[d, "G"]
    \\
    \CC
    \ar[r, "S_2"]
    & \CD
    \ar[r, "R_2"]
    & \CC,
  \end{tikzcd} \]
  and assume that $F : \CB \to \CD$ is a split weak equivalence.

  Take a term $a : \CC.\Tm(G(\Gamma),G(A))$.
  Since $F$ is a weak equivalence, there is $a_0 : \CB.\Tm(S_1(\Gamma),S_1(A))$ and an identification $p_0 : \CD.\Tm(S_2(G(\Gamma)), \Id(F(a_0), S_2(a)))$.
  Then $R_1(a_0) : \CA.\Tm(\Gamma,A)$ and $R_2(p_0) : \CC.\Tm(G(\Gamma), \Id(G(R_1(a_0)), a))$ is an identification witnessing the fact that $R_1(a_0)$ is a weak lift of $a$.
  Thus $G$ is a weak equivalence.
\end{proof}

\subsection{Displayed families}

We now describe displayed HoTT-families, which should be thought as the motives and methods of the induction principle that we will use to prove homotopy canonicity.
Displayed HoTT-families correspond to the notion of displayed higher-order model from~\textcite{InternalSconing}.

\begin{defi}
  A \defemph{displayed cumulative family} $\MM^\bullet$ over a model $\CM$ consists of the following components:
  \begin{alignat*}{1}
    & \Ty_n^\bullet : \CM.\Ty_n(1_\CM) \to \SSet, \\
    & \Tm_n^\bullet : \Ty_n^\bullet(A) \to \CM.\Tm_n(1_\CM,A) \to \SSet, \\
    & \Lift_n^\bullet : \Ty_n^\bullet(A) \to \Ty_{n+1}^\bullet(\Lift_n(A)), \\
    & \lift^\bullet : \Tm_n^\bullet(A^\bullet,a) \cong \Ty_{i+1}^\bullet(\Lift_n^\bullet(A^\bullet),\lift(a)), \\
    & \UU_n^\bullet : \Ty_{n+1}^\bullet(\UU_n), \\
    & \El^\bullet : \Tm^\bullet(\UU_n^\bullet,A) \cong \Ty_n^\bullet(\El(A)).
  \end{alignat*}

  A \defemph{displayed MLTT-family} is a displayed cumulative family together with displayed $\Pi$-types with function extensionality, $\Sigma$-types, $\Unit$-types, $\Id$-types, boolean types, empty types and $W$-types.
\end{defi}

We can compute the following definitions of displayed contractibility witnesses and equivalences.
\begin{alignat*}{1}
  & \isContr^\bullet(A^\bullet,c) = (x^\bullet : A^\bullet(\fst(c))) \times (\forall y\ (y^\bullet : A^\bullet(y)) \to \Id^\bullet(A^\bullet,x^\bullet,y^\bullet, \app(\snd(c), y))), \\
  & \isEquiv^\bullet(f^\bullet,e) \\
  & \quad = (\forall b\ b^\bullet \to \isContr^\bullet(\lambda (a,p) \mapsto (a^\bullet : A^\bullet(a)) \times (p^\bullet : \Id^\bullet(B^\bullet,f^\bullet(a^\bullet),b^\bullet,p)), \app(e,b))), \\
  & \Equiv^\bullet(A^\bullet,B^\bullet,f) = (f^\bullet : \forall a \to A^\bullet(a) \to B^\bullet(\app(\fst(f),a))) \times \isEquiv^\bullet(f^\bullet,\snd(f)).
\end{alignat*}

\begin{defi}
  A displayed HoTT-family is a displayed MLTT-family together with a displayed univalence structure:
  \begin{alignat*}{1}
    & \ua^\bullet : \forall A, (A^\bullet : \Ty_n^\bullet(A)) \to \isContr^\bullet(\lambda (B,f) \mapsto (B^\bullet : \Ty_n^\bullet(B)) \times \Equiv^\bullet(A^\bullet,B^\bullet,f), \ua(A)).
      \tag*{\qedhere}
  \end{alignat*}
\end{defi}

\subsection{Sconing}

We also recall the sconing operation, also called displayed contextualization, which turns a displayed family into a displayed model.
The purpose of this construction is to allow for the use of the induction principle of the syntax of HoTT (any displayed model over the syntax admits a section).
Strict canonicity for MLTT can be proven using an instance of this construction; we refer the reader to~\textcite{InternalSconing} for more details.

\begin{construction}
  If $\MM^\bullet$ is a displayed HoTT-family over $\CM$, we construct a model $\CScone_{\MM^\bullet}$ displayed over $\CM$.
  \begin{itemize}
  \item An object of $\CScone_{\MM^\bullet}$ displayed over $\Gamma \in \CM$ is a family
    \[ \Gamma^\bullet : \CM(1_\CM,\Gamma) \to \SSet. \]
  \item A morphism of $\CScone_{\MM^\bullet}$ from $\Gamma^\bullet$ to $\Delta^\bullet$ displayed over $f \in \CM(\Gamma,\Delta)$ is a family
    \[ f^\bullet : \forall \gamma \to \Gamma^\bullet(\gamma) \to \Gamma^\bullet(f \circ \gamma). \]
  \item A type of $\CScone_{\MM^\bullet}$ over $\Gamma^\bullet$ displayed over $A : \CM.\Ty_n(\Gamma)$ is a family
    \[ A^\bullet : \forall \gamma \to \Gamma^\bullet(\gamma) \to \Ty^\bullet(A[\gamma]). \]
  \item A term of $\CScone_{\MM^\bullet}$ of type $A^\bullet$ displayed over $a : \CM.\Tm(\Gamma,A)$ is a family
    \[ a^\bullet : \forall \gamma \to (\gamma^\bullet : \Gamma^\bullet(\gamma)) \to \Tm^\bullet(A^\bullet(\gamma^\bullet), a[\gamma]). \]
  \item The substitution actions on types and terms are defined by function composition:
    \begin{alignat*}{1}
      & A^\bullet[f^\bullet] \triangleq \lambda \gamma^\bullet \mapsto A^\bullet(f^\bullet(\gamma^\bullet)), \\
      & a^\bullet[f^\bullet] \triangleq \lambda \gamma^\bullet \mapsto a^\bullet(f^\bullet(\gamma^\bullet)).
    \end{alignat*}
  \item The displayed empty context $\diamond^\bullet$ and extended contexts are given by singleton sets and dependent sums:
    \begin{alignat*}{1}
      & \diamond^\bullet \triangleq \lambda \_ \mapsto \{\star\}, \\
      & (\Gamma^\bullet.A^\bullet) \triangleq \lambda (\gamma,a) \mapsto (\gamma^\bullet : \Gamma^\bullet(\gamma)) \times (a^\bullet : A^\bullet(\gamma^\bullet,a)).
    \end{alignat*}
  \item All type-theoretic structures are defined pointwise using the corresponding operation from $\MM^\bullet$.
    \begin{alignat*}{1}
      & \CScone_{\MM^\bullet}.\Unit(\Gamma^\bullet) \triangleq \lambda \gamma^\bullet \mapsto \Unit^\bullet, \\
      & \CScone_{\MM^\bullet}.\Pi(\Gamma^\bullet,A^\bullet,B^\bullet) \triangleq \lambda \gamma^\bullet \mapsto \Pi^\bullet(A^\bullet(\gamma^\bullet), \lambda a^\bullet \mapsto B^\bullet(\gamma^\bullet,a^\bullet)), \\
      & \CScone_{\MM^\bullet}.\lam(\Gamma^\bullet,b^\bullet) \triangleq \lambda \gamma^\bullet \mapsto \lam^\bullet(\lambda a^\bullet \mapsto B^\bullet(\gamma^\bullet,a^\bullet)), \\
      & \CScone_{\MM^\bullet}.\ua(\Gamma^\bullet,A^\bullet) \triangleq \lambda \gamma^\bullet \mapsto \ua^\bullet(A^\bullet(\gamma^\bullet)), \\
      & \dots
    \end{alignat*}
  \item All naturality conditions follow simply from associativity of function composition.
    \qedhere
  \end{itemize}
\end{construction}

\subsection{Relational equivalences}\label{ssec:rel_equivs}

Let $\MM$ be a HoTT-family.
We define relational equivalences (also known as one-to-one correspondences) and reflexivity structures, which will be used to define the path and reflexive-loop models of HoTT.
Relational equivalences are equivalent to other definitions of equivalences in HoTT (\eg half-adjoint equivalences).
A self-equivalence has a reflexivity structure when it is homotopic to the identity equivalence.
The definition of these structures as families of types together with contractibility conditions permits the definition of models corresponding to parametricity translations (the families of types are then seen as logical relations).

\begin{defi}
  Given types $A_1, A_2 : \MM_n$, a \defemph{relational equivalence} $A_e : \RelEquiv(A_1, A_2)$ consists of a type-valued relation
  \begin{alignat*}{1}
    & A_e : A_1 \to A_2 \to \MM_n,
  \end{alignat*}
  and families of contractibility proofs witnessing that $A_e$ is functional in both directions
  \begin{alignat*}{1}
    & A_e.\funr : (a_1 : A_1) \to \isContr((a_2 : A_2) \times A_e(a_1,a_2)), \\
    & A_e.\funl : (a_2 : A_2) \to \isContr((a_1 : A_1) \times A_e(a_1,a_2)).
      \tag*{\qedhere}
  \end{alignat*}
\end{defi}

\begin{defi}
  A \defemph{reflexivity structure} $A_r : \isRefl(A_e)$ over an equivalence $A_e : \RelEquiv(A,A)$ consists of a family
  \begin{alignat*}{1}
    & A_r : (a : A) \to A_e(a,a) \to \MM_n,
  \end{alignat*}
  along with a family of contractibility proofs witnessing the unique existence of a reflexivity loop
  \begin{alignat*}{1}
    & A_r.\refl : (a : A) \to \isContr((a_e : A_e(a,a)) \times A_r(a,a_e)).
      \tag*{\qedhere}
  \end{alignat*}
\end{defi}

\begin{construction}
  Given a relational equivalence $A_e : \RelEquiv(A_1, A_2)$ and elements $x_e : A_e(x_1,x_2)$ and $y_e : A_e(y_1,y_2)$, there is a relational equivalence $\Id^{\RelEquiv}(A_e,x_e,y_e) : \RelEquiv(\Id_{A_1}(x_1,y_1), \Id_{A_2}(x_2,y_2))$, defined by
  \[ \Id^{\RelEquiv}(A_e,x_e,y_e) \triangleq \lambda \refl\ \refl \mapsto \Id_{A_e(x_1,x_2)}(x_e,y_e).
  \]

  When $A_e$ is a self-equivalence with a reflexivity structure $A_r : \isRefl(A_e)$ and we have elements $x_r : A_r(x,x_e)$ and $y_r : A_r(y,y_e)$, there is a reflexivity structure $\Id^{\isRefl}(A_r,x_r,y_r) : \isRefl(\Id^{\RelEquiv}(A_e,x_e,y_e))$, defined by
  \[ \Id^{\isRefl}(A_r,x_r,y_r) \triangleq \lambda \refl\ \refl \mapsto \Id_{A_r(x,x_e)}(x_r,y_r).
    \tag*{\qedhere{}}
  \]
\end{construction}

\begin{construction}
  The universe $\UU_n$ has a reflexive relational equivalence, given by:
  \begin{alignat*}{1}
    & \UU^\RelEquiv \triangleq \lambda A\ B \to \RelEquiv(A,B), \\
    & \UU^\isRefl \triangleq \lambda A\ E \to \isRefl(A,E).
  \end{alignat*}
  The contractibility conditions follow from univalence.
\end{construction}

Relational equivalences and reflexivity structures are also preserved by the other type formers ($\Sigma$-, $\Pi$-, $W$-, boolean, empty and unit types).
For details see the Agda formalization.

\subsection{Path and reflexive loop models}\label{ssec:path_and_reflloop_models}

We construct path and reflexive-loop models of HoTT.
These are instances of homotopical inverse diagram models~\parencite{HomotopicalInverseDiagrams}, indexed respectively by the homotopical inverse categories
\[ \begin{tikzcd}
    & E
    \ar[ld, "\pi_1"']
    \ar[rd, "\pi_2"]
    \\
    V_1
    && V_2
  \end{tikzcd} \text{and}
 \begin{tikzcd}
    & R \ar[r]
    \ar[r, "p_e"]
    & E
    \ar[r, "p_1", shift left]
    \ar[r, "p_2"', shift right]
    & V.
  \end{tikzcd} \]
They are also closely related to the univalent parametricity translation of~\textcite{MarriageUnivalenceParametricity}.

We only define these models for a democratic base model, although variants of the constructions exist for an arbitrary base model.

\begin{construction}
  For any democratic model $\CM$ of HoTT, we construct another model $\Path_\CM$, called the \defemph{path-model} of $\CM$.
  We define it as a displayed model over $\CM \times \CM$.
  \begin{itemize}
  \item An object of $\Path_\CM$ displayed over $\Gamma_1,\Gamma_2$ is an equivalence
    \begin{alignat*}{1}
      & \Gamma_e \Colon \RelEquiv(\Gamma_1, \Gamma_2).
    \end{alignat*}
  \item A morphism of $\Path_\CM$ from $\Gamma_e$ to $\Delta_e$ displayed over $f_1 \Colon \Gamma_1 \to \Delta_1$ and $f_2 \Colon \Gamma_2 \to \Delta_2$ is a function
    \begin{alignat*}{1}
      & f_e \Colon \Gamma_e(\gamma_1,\gamma_2) \to \Delta_e(f_1(\gamma_1),f_2(\gamma_2)).
    \end{alignat*}
  \item A type of $\Path_\CM$ over $\Gamma_e$ and displayed over $A_1 \Colon \Gamma_1 \to \UU_n$ and $B_1 \Colon \Gamma_2 \to \UU_n$ is a family of equivalences
    \begin{alignat*}{1}
      & A_e \Colon (\gamma_e:\Gamma_e(\gamma_1,\gamma_2)) \to \RelEquiv(A_1(\gamma_1), A_2(\gamma_2)).
    \end{alignat*}
  \item A term of $\Path_\CM$ of type $A_e$ and displayed over $a_1 \Colon (\gamma_1:\Gamma_1) \to A_1(\gamma_1)$ and $B_1 \Colon (\gamma_2:\Gamma_2) \to A_2(\gamma_2)$ is a family
    \begin{alignat*}{1}
      & a_e \Colon (\gamma_e:\Gamma_e(\gamma_1,\gamma_2)) \to A_e(a_1(\gamma_1),a_2(\gamma_2)).
    \end{alignat*}
  \item The type formers are interpreted pointwise over $\gamma_e : \Gamma_e(\gamma_1,\gamma_2)$ using the constructions of~\autoref{ssec:rel_equivs}.
    For example,
    \begin{alignat*}{1}
      & \Path_\CM.\Id_{A_e}(x_e,y_e) \triangleq \lambda \gamma_e \mapsto \Id^\RelEquiv(A_e(\gamma_e),x_e(\gamma_e),y_e(\gamma_e)).
    \end{alignat*}
  \item The rest of the structure corresponds to a standard binary parametricity construction.
    \qedhere
  \end{itemize}
\end{construction}

The \defemph{loop-model} $\Loop_\CM$ of a model $\CM$ is the pullback
\[ \begin{tikzcd}
    \Loop_\CM
    \ar[rd, phantom, very near start, "\lrcorner"]
    \ar[r]
    \ar[d, "\pi", two heads]
    &
    \Path_\CM
    \ar[d, "{\angles{\pi_1,\pi_2}}", two heads]
    \\
    \CM
    \ar[r, "\angles{\id,\id}"]
    &
    \CM \times \CM
  \end{tikzcd} \]

\begin{construction}
  For any democratic model $\CM$ of HoTT, we construct another model $\ReflLoop_\CM$, called the \defemph{reflexive-loop-model} of $\CM$.
  We define it as a displayed model over $\Loop_\CM$.
  \begin{itemize}
  \item An object of $\ReflLoop_\CM$ displayed over $\Gamma,\Gamma_e$ is a reflexivity structure
    \begin{alignat*}{1}
      & \Gamma_r \Colon \isRefl(\Gamma_e).
    \end{alignat*}
  \item A morphism of $\ReflLoop_\CM$ from $\Gamma_r$ to $\Delta_r$ and displayed over $f,f_e$ is a map
    \begin{alignat*}{1}
      & f_r \Colon \Gamma_r(\gamma,\gamma_e) \to \Delta_r(f(\gamma),f_e(\gamma_e)).
    \end{alignat*}
  \item A type of $\ReflLoop_\CM$ displayed over $A \Colon \Gamma \to \UU_n$ and $A_e \Colon \forall \gamma\ \gamma_e \to \RelEquiv(A(\gamma), A(\gamma))$ is a family of reflexivity structures
    \begin{alignat*}{1}
      & A_r \Colon (\gamma_r:\Gamma_r(\gamma,\gamma_e)) \to \isRefl(A_e(\gamma_e)).
    \end{alignat*}
  \item A term of $\ReflLoop_\CM$ of type $A_r$ displayed over $a \Colon (\gamma:\Gamma) \to A(\gamma)$ and $a_e \Colon \forall \gamma\ \gamma_e \to A_e(\gamma_e,a(\gamma),a(\gamma))$ is a family of reflexivity structures
    \begin{alignat*}{1}
      & a_r \Colon (\gamma_r:\Gamma_r(\gamma,\gamma_e)) \to A_r(a(\gamma),a_e(\gamma_e)).
    \end{alignat*}
  \item The type formers are interpreted using the constructions of~\autoref{ssec:rel_equivs}.
  \item The rest of the structure corresponds to a standard parametricity construction.
    \qedhere
  \end{itemize}
\end{construction}

\begin{prop}\label{prop:hott_pi12_split_fib}
  The projection $\angles{\pi_1,\pi_2} : \Path_\CM \to \CM \times \CM$ is a split fibration.
\end{prop}
\begin{proof}
  We first prove that $\angles{\pi_1,\pi_2} : \Path_\CM \to \CM \times \CM$ satisfies the identification lifting property.
  Take a term $x$ of $\Path_\CM$. It consists of an equivalence $\Gamma_e \Colon \RelEquiv(\Gamma_1, \Gamma_2)$, a family $A_e \Colon \forall \gamma_e \to \RelEquiv(A_1(\gamma_1), A_2(\gamma_2))$ of equivalences and a family $x_e \Colon \forall \gamma_e \to A_e(\gamma_e,x_1(\gamma_1),x_2(\gamma_2))$.
  Take an identification $p$ in $\CM \times \CM$ between $\angles{\pi_1,\pi_2}(x)$ and a term $y$.
  It consists of $p_1 \Colon \forall \gamma_1 \to \Id_{A_1(\gamma_1)}(x_1(\gamma_1),y_1(\gamma_1))$ and $p_2 \Colon \forall \gamma_2 \to \Id_{A_2(\gamma_2)}(x_2(\gamma_2),y_2(\gamma_2))$.
  We then define $y_e \Colon \forall \gamma_e \to A_e(\gamma_e,y_1(\gamma_1),y_2(\gamma_2))$ by transporting $x_e$ over $p_1$ and $p_2$.
  We obtain $p_e \Colon \forall \gamma_e \to \Id^{\RelEquiv}(A_e,x_e,y_e)(p_1(\gamma_1),p_2(\gamma_2))$ as a witness of the fact that $y_e$ is a transport of $x_e$ over $p_1$ and $p_2$.
  Then $(y_e,p_e)$ is a lift of the identification $(y,p)$ against $\angles{\pi_1,\pi_2}$.
  Thus $\angles{\pi_1,\pi_2}$ is a split fibration.
\end{proof}

\begin{prop}\label{prop:hott_pi1_split_tfib}
  The projections $\pi_1,\pi_2 : \Path_\CM \to \CM$ are split trivial fibrations.
\end{prop}
\begin{proof}
  We prove that $\pi_1$ satisfies the term lifting property, the case of $\pi_2$ is symmetric.
  Take a type $A$ in $\Path_\CM$.
  It consists of an equivalence $\Gamma_e \Colon \RelEquiv(\Gamma_1, \Gamma_2)$ and a family $A_e \Colon \forall \gamma_e \to \RelEquiv(A_1(\gamma_1), A_2(\gamma_2))$ of equivalences.
  Take a term of type $\pi_1(A)$ in $\CM$, \ie a term $x_1 \Colon \forall \gamma_1 \to A_1(\gamma_1)$.
  We then define a term $x_2 \Colon \forall \gamma_2 \to A_2(\gamma_2)$ by transport over the equivalences $\Gamma_e$ and $A_e$.
  We have an element $x_e \Colon \forall \gamma_e \to A_e(\gamma_e,x_1(\gamma_1),x_2(\gamma_2))$ witnessing that $x_2$ was defined by transporting $x_1$.
  Then $(x_1,x_2,x_e)$ is a lift of the term $x_1$ along $\pi_1$.
  Thus $\pi_1$ is a split trivial fibration.
\end{proof}

\begin{prop}
  The projection $\pi_e : \ReflLoop_\CM \to \Loop_\CM$ is a split fibration.
\end{prop}
\begin{proof}
  Similar to~\autoref{prop:hott_pi12_split_fib}.
\end{proof}

\begin{prop}
  The projection $\pi : \ReflLoop_\CM \to \CM$ is a split trivial fibration.
\end{prop}
\begin{proof}
  Similar to~\autoref{prop:hott_pi1_split_tfib}.
\end{proof}

\begin{prop}
  The constructions of $\Path_\CM$ and $\ReflLoop_\CM$ are functorial in $\CM$.
\end{prop}
\begin{proof}
  This follows from the fact that all components of $\Path_\CM$ and $\ReflLoop_\CM$ are expressed in the ``language of HoTT'', \eg as finite limits of components of $\CM$.

  This can be stated more precisely using functorial semantics: the democratic models of HoTT are algebras for an essentially algebraic theory $\Th_{\mathsf{HoTT}}^\dem$.
  The model $\CM$ is a left exact functor $\CM : \Th_{\mathsf{HoTT}}^\dem \to \CSet$.
  We then observe that $\Path_\CM = \CM \circ P$ and $\ReflLoop_\CM = \CM \circ R$ for some left exact functors $P, R : \Th_{\mathsf{HoTT}}^\dem \to \Th_{\mathsf{HoTT}}^\dem$ (which can be constructed using the universal property of $\Th_{\mathsf{HoTT}}^\dem$).
  The functoriality is then immediate.
\end{proof}

\begin{prop}\label{prop:hott_tcof_is_weq}
  Let $F : \CM \to \CN$ be a morphism between democratic models of HoTT.
  If $\pi : \ReflLoop_\CM \to \CM$ admits a section $r$ and $F$ is an algebraic trivial cofibration, then $F$ is a split weak equivalence.
\end{prop}
\begin{proof}
  Same proof as~\autoref{prop:cat_tcof_is_weq}.
\end{proof}


\section{Strict Rezk completions for models of HoTT}\label{sec:completion_hott}

For most of this section, we work internally to $\CcSet$.
The structure of the construction is the same as in~\autoref{sec:completion_categories},
and most of the lemmas have the exact same proofs.
In such case we omit the proof and refer to the corresponding proof in~\autoref{sec:completion_categories}.

We say that a model $\CM$ of HoTT has fibrant components if for every $A : \CM.\Ty_n(\Gamma)$, the set $\CM.\Tm(\Gamma,A)$ is fibrant.
Note that as a special case, the sets $\CM.\Ty_n(\Gamma)$ are fibrant, since $\CM.\Ty_n(\Gamma) \cong \CM.\Tm(\Gamma,\UU_n)$.

\begin{defi}
  We say that a model $\CM$ of HoTT with fibrant components is \defemph{complete} when:
  \begin{itemize}
  \item For every term $x : \CM.\Tm(\Gamma,A)$, the fibrant set $(y : \CM.\Tm(\Gamma,A)) \times (p : \CM.\Tm(\Gamma,\Id_A(x,y)))$ is contractible.
    \qedhere
  \end{itemize}
\end{defi}

\begin{defi}
  A \defemph{strict Rezk completion} of a global model $\CM$ of HoTT with fibrant components is a global complete model $\overline{\CM}$, along with a morphism $i : \CM \to \overline{\CM}$ such that the external morphism $1^\ast_\square(i) : 1^\ast_\square(\CM) \to 1^\ast_\square(\overline{\CM})$ is a split weak equivalence of models of HoTT.
\end{defi}

\begin{defi}\label{def:hott_glue_structure}
  A \defemph{$\Cof$-fibrancy structure} over a model $\CM$ consists of:
  \begin{itemize}
  \item For every term $x : \CM.\Tm(\Gamma,A)$, an extension structure
    \[ \ext_\Tm(x) : \HasExt((y : \CM.\Tm(\Gamma,A)) \times (p : \CM.\Tm(\Gamma,\Id_A(x,y))). \]
    \qedhere
  \end{itemize}
\end{defi}

A $\Cof$-fibrancy structure can be decomposed into operations
\begin{alignat*}{1}
  & \Glue_{\Tm} : (\Gamma \in \CM) (A : \CM.\Ty(\Gamma)) (x : \CM.\Tm(\Gamma,A)) (\alpha : \Cof) (y : [\alpha] \to \CM.\Tm(\Gamma,A)) \\
  & \quad \to (p : [\alpha] \to \CM.\Tm(\Gamma,\Id_A(x,y))) \to \{ \CM.\Tm(\Gamma,A) \mid \alpha \hra y \}, \\
  & \glue_{\Tm} : (\Gamma \in \CM) (A : \CM.\Ty(\Gamma)) (x : \CM.\Tm(\Gamma,A)) (\alpha : \Cof) (y : [\alpha] \to \CM.\Tm(\Gamma,A)) \\
  & \quad \to (p : [\alpha] \to \CM.\Tm(\Gamma,\Id_A(x,y))) \to \{ \CM.\Tm(\Gamma,\Id_A(x,\Glue_\Tm(x,y,p))) \mid \alpha \hra p \},
\end{alignat*}
with $\angles{\Glue_\Tm,\glue_\Tm} = \ext_\Tm$.

\begin{lem}\label{lem:hott_glue_contr}
  If $\CM$ is $\Cof$-fibrant and a type $A : \CM.\Ty(\Gamma)$ is contractible, then $\CM.\Tm(\Gamma,A)$ is trivially fibrant.
\end{lem}
\begin{proof}
  Write $c : \CM.\Tm(\Gamma,A)$ for the center of contraction of $A$.
  Given any $y : \CM.\Tm(\Gamma,A)$, we have an identification $p_A(y) : \CM.\Tm(\Gamma,\Id_A(c,y))$.
  
  Take a partial element $x_0 : [\alpha] \to \CM.\Tm(\Gamma,A)$.
  We then have a total element
  \[ x \triangleq \Glue_\Tm(c, [\alpha \mapsto (x_0, p_A(x_0))]) \]
  extending $x_0$.

  This equips $\CM.\Tm(\Gamma,A)$ with an extension structure, as needed.
\end{proof}

\begin{prop}\label{prop:hott_glue_complete}
  If model $\CM$ with fibrant components is $\Cof$-fibrant, then it is complete.
\end{prop}
\begin{proof}
  By~\autoref{prop:fibrant_contractibility_iff_ext}.
\end{proof}

Let $\CM$ be a global democratic model of HoTT.
Similarly to the case of categories, we want to prove that the $\Cof$-fibrant replacement $\overline{\CM}$ of $\CM$ a strict Rezk completion.
In order to use~\autoref{lem:fibrancy_from_reflgraph}, we need to show that pseudo-reflexive graphs arising from the pseudo-reflexive graph object
\[ \begin{tikzcd}
    \ReflLoop_{\overline{\CM}}
    \ar[r, "\pi_e"]
    & \Path_{\overline{\CM}}
    \ar[r, "\pi_1", shift left]
    \ar[r, "\pi_2"', shift right]
    & {\overline{\CM}}
  \end{tikzcd} \]
have weak coercion operations and are homotopical.

\begin{defi}
  A \defemph{weak coercion structure} over a line $\Gamma : \MI \to \Ob_\CM$ consists of families
  \begin{alignat*}{1}
    & \wcoe^{r \to s}_\Gamma \Colon \RelEquiv(\Gamma(r), \Gamma(s)), \\
    & \wcoh^{r}_\Gamma \Colon \isRefl(\wcoe^{r \to r}_\Gamma)
  \end{alignat*}
  of equivalences and reflexivity structures.

  Given weak coercion structures $\wcoe_\Gamma$ and $\wcoe_\Delta$, a \defemph{weak coercion structure} over a line $f : (i : \MI) \to \CM(\Gamma(i),\Delta(i))$ consists of families
  \begin{alignat*}{1}
    & \wcoe^{r \to s}_f \Colon \wcoe_\Gamma^{r \to s}(\gamma_1,\gamma_2) \to \wcoe_\Delta^{r \to s}(f(\gamma_1),f(\gamma_2)), \\
    & \wcoh^{r}_f \Colon \wcoh_\Gamma^r(\gamma,\gamma_e) \to \wcoh^r_\Delta(f(\gamma),\wcoe^{r \to r}_f(\gamma_e)).
  \end{alignat*}
  
  Given a weak coercion structure $\wcoe_\Gamma$, a \defemph{weak coercion structure} over a line $f : (i : \MI) \to \CM.\Ty(\Gamma(i))$ consists of families
  \begin{alignat*}{1}
    & \wcoe^{r \to s}_A \Colon \wcoe_\Gamma^{r \to s}(\gamma_1,\gamma_2) \to \RelEquiv(A(r,\gamma_1), A(s,\gamma_2)), \\
    & \wcoh^{r}_A \Colon \wcoh_\Gamma^r(\gamma,\gamma_e) \to \isRefl(\wcoe^{r \to r}_A(\gamma_e)).
  \end{alignat*}

  Given weak coercion structures $\wcoe_\Gamma$ and $\wcoe_A$, a \defemph{weak coercion structure} over a line $a : (i : \MI) \to \CM.\Tm(\Gamma(i),A(i))$ consists of families
  \begin{alignat*}{1}
    & \wcoe^{r \to s}_a \Colon (\gamma_e : \wcoe_\Gamma^{r \to s}(\gamma_1,\gamma_2)) \to \wcoe_A^{r \to s}(\gamma_e,a(r),a(s)), \\
    & \wcoh^{r}_a \Colon (\gamma_r : \wcoh_\Gamma^r(\gamma,\gamma_e)) \to \wcoh_A^r(\gamma_r,\wcoe^{r \to r}_a(\gamma_e)).
      \tag*{\qedhere}
  \end{alignat*}
\end{defi}

\begin{construction}
  We construct a displayed model $\HasWCoe^\CM$ over $\CM^\MI$.
  As described in~\autoref{rem:cat_haswcoe_limit}, we construct it as the limit of the diagram
  \begin{alignat*}{1}
    & \mathsf{base} \mapsto \CM^\MI, \\
    & \mathsf{path}(r,s) \mapsto \Path_\CM[\angles{-_r,-_s}], \\
    & \mathsf{refl}\text{-}\mathsf{loop}(r) \mapsto \ReflLoop_\CM[\angles{-_r}],
  \end{alignat*}
  over the diagram shape consisting of objects $\mathsf{base}$, $\mathsf{path}(r,s)$ for $r,s : \MI$ and $\mathsf{refl}\text{-}\mathsf{loop}(r)$ for $r : \MI$, such that $\mathsf{base}$ is terminal and with morphisms $\mathsf{refl}\text{-}\mathsf{loop}(r) \to \mathsf{path}(r,r)$ for $r:\MI$.

  We can verify by unfolding the definition that the displayed objects, morphisms, types and terms of this model are weak coercion structures over lines of objects, morphisms, types and terms of $\CM$.
\end{construction}

\begin{prop}\label{prop:hott_haswcoe_tfib}
  If a model $\CM$ has fibrant components, then the projection $\HasWCoe^\CM \to \CM^\MI$ is a split trivial fibration.
\end{prop}
\begin{proof}
  Same proof as~\autoref{prop:cat_haswcoe_tfib}.
\end{proof}

Now assume that $\CM$ is a global algebraically cofibrant model of HoTT with fibrant components.

\begin{lem}
  Any algebraically cofibrant model is democratic; in particular, the model $\CM$ is democratic.
\end{lem}
\begin{proof}
  Write $\dem(\CM)$ for the democratic core of $\CM$.
  An object of $\dem(\CM)$ consists of an object $\Gamma \in \CM$ together with a closed type $K_\Gamma$ an isomorphism $\Gamma \cong 1.K_\Gamma$.
  The rest of the structure is inherited from $\CM$.
  Then the projection $\pi : \dem(\CM) \to \CM$ is a split trivial fibration.
  Since $\CM$ is algebraically cofibrant, the projection admits a section, which witnesses the democracy of $\CM$.
\end{proof}

\begin{construction}\label{constr:strict_rezk_hott}
  We write $\overline{\CM}$ for the $\Cof$-fibrant replacement of $\CM$, \ie the model freely generated by a morphism $i : \CM \to \overline{\CM}$ and a $\Cof$-fibrancy structure.
\end{construction}

\begin{prop}
  The model $\overline{\CM}$ is democratic.
\end{prop}
\begin{proof}
  This follows from the fact that $\overline{\CM}$ is obtained from the democratic model $\CM$ by only adding new terms and equations between terms.
\end{proof}

\begin{lem}\label{lem:hott_coffib_commutes_expi}
  The map $i^\MI : \CM^\MI \to \overline{\CM}^\MI$ exhibits $\overline{\CM}^\MI$ as a $\Cof$-fibrant replacement of $\CM^\MI$.
\end{lem}
\begin{proof}
  Same proof as~\autoref{lem:cat_coffib_commutes_expi}.
\end{proof}

\begin{prop}\label{prop:hott_haswcoe_glue}
  The displayed model $\HasWCoe^{\overline{\CM}}$ can be equipped with a displayed $\Cof$-fibrancy structure.
\end{prop}
\begin{proof}
  Take a cofibration $\alpha$ and lines $\Gamma : \MI \to \Ob_{\overline{\CM}}$, $A : (i:\MI) \to \overline{\CM}.\Ty(\Gamma(i))$, $x : (i:\MI) \to \overline{\CM}.\Tm(\Gamma(i),A(i))$, $y : [\alpha] \to (i : \MI) \to \overline{\CM}.\Tm(\Gamma(i),A(i))$ and $p : [\alpha] \to (i : \MI) \to \overline{\CM}.\Tm(\Gamma(i),\Id_{A(i)}(x(i),y(i)))$.
  
  We need to define weak coercion structures over the lines $G(-) = \Glue_\Tm(x(-),y(-),e(-))$ and $g(-) = \glue_\Tm(x(-),y(-),e(-))$.
  They should match with the weak coercion structures of $y$ and $p$ under $\alpha$.

  Over the context $(\gamma_1:\Gamma).(\gamma_2.\Gamma).(\gamma_e:\wcoe^{r\to s}_\Gamma(\gamma_1,\gamma_2))$, the type
  \[ (G_e : \wcoe^{r \to s}_A(\gamma_e,G(r),G(s))) \times (g_e : \Id^\RelEquiv(\wcoe_A^{r \to s}(\gamma_e),\wcoe_x^{r \to s}(\gamma_e),G_e,g(r),g(s)))
  \]
  is contractible (this follows from the definition of $\Id^\RelEquiv$).
  By~\autoref{lem:hott_glue_contr}, this type has a term $\angles{\wcoe^{r \to s}_G(\gamma_e),\wcoe^{r \to s}_g(\gamma_e)}$ that restricts to $\angles{\wcoe^{r \to s}_y(\gamma_e),\wcoe^{r \to s}_p(\gamma_e)}$ under $\alpha$.

  Over the context $(\gamma:\Gamma).(\gamma_e.\wcoe^{r\to s}_\Gamma(\gamma,\gamma)).(\gamma_r:\wcoh^{r}_\Gamma(\gamma,\gamma_e))$, the type
  \[ (G_r : \wcoh^{r}_A(\gamma_r,\wcoe^{r \to r}_G(\gamma_e))) \times (g_e : \Id^\isRefl(\wcoh_A^{r}(\gamma_r),\wcoh_x^{r}(\gamma_r),G_r,\wcoe^{r\to r}_g(\gamma_r)))
  \]
  is contractible (this follows from the definition of $\Id^\isRefl$).
  By~\autoref{lem:hott_glue_contr}, this type has a term $\angles{\wcoh^{r}_G(\gamma_r),\wcoh^{r}_g(\gamma_r)}$ that restricts to $\angles{\wcoh^{r}_y(\gamma_r),\wcoh^{r}_p(\gamma_r)}$ under $\alpha$.
\end{proof}

\begin{prop}\label{prop:hott_sect_haswcoe}
  The displayed model $\HasWCoe^{\overline{\CM}} \to \overline{\CM}^\MI$ admits a global section.
\end{prop}
\begin{proof}
  By~\autoref{prop:hott_haswcoe_tfib}, $\HasWCoe^{\CC} \to \CC^\MI$ is a split trivial fibration.
  Since $\CC$ is algebraically cofibrant, it admits a section.
  By composing this section with $\HasWCoe^\CC \to \HasWCoe^{\overline{\CC}}$, we obtain a map $\CC^\MI \to \HasWCoe^{\overline{\CC}}$ displayed over $i^\MI : \CC^\MI \to \overline{\CC}^\MI$.

  By combining this with~\autoref{prop:hott_haswcoe_glue}, we can use the universal property of $\overline{\CC}^\MI$ from~\autoref{lem:hott_coffib_commutes_expi} to obtain a section of $\HasWCoe^{\overline{\CC}} \to \overline{\CC}^\MI$.
\end{proof}

As a consequence, every object, morphism, type or term of $\overline{\CM}$ can be equipped with a weak coercion structure $\wcoe_{-}$.

\begin{prop}\label{prop:strict_rezk_hott_fibrant}
  The model $\overline{\CM}$ has fibrant components.
\end{prop}
\begin{proof}
  We use~\autoref{lem:fibrancy_from_reflgraph} for $A = (\Gamma : \Ob_{\overline{\CM}}) \times \overline{\CM}.\Ty(\Gamma)$ and $B(\Gamma,X) = \overline{\CM}.\Tm(\Gamma,X)$.
  The families $E_A$, $R_A$, $E_B$ and $R_B$ are the corresponding finite limits of components of $\Path_{\overline{\CC}}$ and $\ReflLoop_{\overline{\CC}}$, namely:
  \begin{alignat*}{1}
    & E_A((\Gamma_1,X_1),(\Gamma_2,X_2)) = (\Gamma_e : \RelEquiv(\Gamma_1,\Gamma_2)) \times (X_e : \forall \gamma_e \to \RelEquiv(X_1(\gamma_1), X_2(\gamma_2))), \\
    & R_A((\Gamma,X),(\Gamma_e,X_e)) = (\Gamma_r : \isRefl(\Gamma_e)) \times (X_e : \forall \gamma_r \to \isRefl(X_e(\gamma_e))), \\
    & E_B((\Gamma_e,X_e),a_1,a_2) = \forall \gamma_e \to X_e(\gamma_e,a_1,a_2), \\
    & R_B((\Gamma_r,X_r),a,a_e) = \forall \gamma_r \to X_r(\gamma_r,a,a_e).
  \end{alignat*}
  By~\autoref{prop:hott_sect_haswcoe}, we have the required operations $\wcoe$ and $\wcoh$.
  
  It remains to construct extension structures
  \begin{alignat*}{1}
    & \forall (\Gamma_e,X_e)\ a_1 \to \HasExt((a_2 : \forall \gamma_2 \to X_2(\gamma_2)) \times (\forall \gamma_e \to X_e(\gamma_e,a_1(\gamma_1),a_2))), \\
    & \forall (\Gamma_r,X_r)\ a \to \HasExt((a_e : \forall \gamma_e \to X_e(\gamma_e,a(\gamma_1),a(\gamma_2))) \times (\forall \gamma_r \to X_r(\gamma_r,a(\gamma),a_e))).
  \end{alignat*}
  We use~\autoref{lem:hott_glue_contr} in both cases (relying on $\Pi$-types to move to the empty context).
  The contractibility of $(a_2 : \forall \gamma_2 \to X_2(\gamma_2)) \times (\forall \gamma_e \to X_e(\gamma_e,a_1(\gamma_1),a_2))$ follows from $\Gamma_e.\funl$ and $X_e.\funr$.
  The contractibility of $(a_e : \forall \gamma_e \to X_e(\gamma_e,a(\gamma_1),a(\gamma_2))) \times (\forall \gamma_r \to X_r(\gamma_r,a(\gamma),a_e))$ relies on $\Gamma_e.\funl$, $\Gamma_r.\refl$ and $X_r.\refl$.

  In fact the contractibility witnesses can be chosen to be $\Pi^\RelEquiv(\Gamma_e,X_e).\funr$ and $\Pi^\isRefl(\Gamma_r,X_r).\refl$.
\end{proof}

\begin{lem}\label{lem:hott_ext_cof_fibrant_is_fibrant}
  The external model $1^\ast_\square(\overline{\CM})$ has the universal property of the fibrant replacement of $1^\ast_\square(\CM)$.
\end{lem}
\begin{proof}
  Same as~\autoref{lem:ext_cof_fibrant_is_fibrant}.
\end{proof}

\begin{prop}\label{prop:strict_rezk_completion_hott_weq}
  The morphism $1^\ast_\square(i) : 1^\ast_\square(\CM) \to 1^\ast_\square(\overline{\CM})$ is a split weak equivalence.
\end{prop}
\begin{proof}
  Same as the proof of~\autoref{prop:strict_rezk_completion_cat_weq}, relying on~\autoref{prop:hott_tcof_is_weq} and~\autoref{lem:hott_ext_cof_fibrant_is_fibrant}.
\end{proof}

\begin{thm}\label{thm:hott_strict_rezk}
  Any global algebraically cofibrant model of HoTT with fibrant components admits a strict Rezk completion.
\end{thm}
\begin{proof}
  We use the model $\overline{\CM}$ defined in~\autoref{constr:strict_rezk_hott}.
  By~\autoref{prop:strict_rezk_hott_fibrant} it has fibrant components.
  By~\autoref{prop:hott_glue_complete} it is complete.
  By~\autoref{prop:strict_rezk_completion_hott_weq} the morphism $1^\ast_\square(i) : 1^\ast_\square(\CM) \to 1^\ast_\square(\overline{\CM})$ is a split weak equivalence.
\end{proof}

\begin{rem}\label{rem:hott_strict_rezk_axioms}
  The model of HoTT freely generated by any number of axioms (closed terms, without any new equations) is algebraically cofibrant.
\end{rem}


\section{Homotopy canonicity}\label{sec:homotopy_canonicity}

We will prove the following theorem.
\begin{thm}[Homotopy canonicity]\label{thm:homotopy_canonicity}
  Let $\CS$ be the initial model of HoTT.
  
  For every closed term $b : \CS.\Tm(1,\BoolTy)$, there is an element of
  \[ \CS.\Tm(1,\Id(b, \true)) + \CS.\Tm(1,\Id(b, \false)). \]
\end{thm}

We work internally to $\CcSet$.

Let $\CS$ be the initial model of HoTT.
Because $\CS$ is the initial algebra of an external generalized algebraic theory, it coincides with the external syntax of HoTT,
\ie the external model $1^\ast_\square(\CS)$ is also initial among external models of HoTT.

Write $\overline{\CS}$ for the Rezk completion of $\CS$, which exists by~\autoref{thm:hott_strict_rezk}, as noted in~\autoref{rem:hott_strict_rezk_axioms}.

\begin{lem}\label{lem:complete_model_contr_internal_iff_external}
  Let $A : \overline{\CS}.\Ty_n(\Gamma)$ be a type.
  If $A$ is contractible in $\overline{\CS}$, then $\overline{\CS}.\Tm(\Gamma,A)$ is contractible.
\end{lem}
\begin{proof}
  Since $\overline{\CS}$ is complete,
  we have equivalences $\overline{\CS}.\Tm(\Gamma,\Id_A(x,y)) \simeq (x \sim y)$.
  
  The type $A$ is contractible in $\overline{\CS}$, which means we have a center of contraction $a_0 : \overline{\CS}.\Tm(\Gamma,A)$ and an identification $p : \overline{\CS}.\Tm(\Gamma.(x:A).(y:A), \Id_A(x,y))$.
  Since $\overline{\CS}$ is complete,
  we have a path $(x \sim y)$ in $\overline{\CS}.\Tm(\Gamma.(x:A).(y:A), \Id_A(x,y))$.
  Now given $x',y' : \overline{\CS}.\Tm(\Gamma,A)$, consider the substitution $\angles{x',y'} : \overline{\CS}(\Gamma, \Gamma.(x:A).(y:A))$.
  We have $x[\angles{x',y'}] \sim y[\angles{x',y'}]$, \ie $x' \sim y'$.

  This shows that $\overline{\CS}.\Tm(\Gamma,A)$ has a center of contraction and is a homotopy proposition.
  It is therefore contractible.
\end{proof}

\begin{rem}
  More generally, $A$ is contractible in $\overline{\CS}$ if and only if $\overline{\CS}.\Tm(\Delta,A[f])$ is contractible for any $f : \Delta \to \Gamma$ (looking at $\Gamma \to \Gamma$ and $\Gamma.A.A \to \Gamma$ suffices).
\end{rem}

The model $\overline{\CS}$ being complete also implies the existence of an equivalence
\[ \idToPath : \overline{\CS}.\Tm(\Gamma,\Id_A(x,y)) \to (x \sim y) \]
which sends $\refl$ to a homotopically constant path.

\subsection{The canonicity model}

In this subsection we define a (large) displayed higher-order model $\MS^\bullet$ over $\overline{\CS}$, which will be used to prove canonicity.
It is very similar to the displayed higher-order model used to prove canonicity for MLTT, except that we use logical predicates valued into fibrant sets.

A displayed type over a closed type $A : \overline{\CS}.\Ty_n(1)$ is a unary logical predicate valued in the universe of $n$-small fibrant sets:
\[ \Ty^\bullet_n(A) \triangleq \overline{\CS}.\Tm(1,A) \to \SSet^{\mathsf{fib}}_n. \]
A displayed term of type $A^\bullet$ over a term $a : \overline{\CS}.\Tm(1,A)$ is an element of the logical predicate $A^\bullet$ at $a$:
\[ \Tm^\bullet(A^\bullet,a) \triangleq A^\bullet(a). \]

\subsubsection{Identity types}

The logical predicate for identity types is defined as the HIT
\begin{alignat*}{1}
  & \Id^\bullet : \forall A\ x\ y\ (A^\bullet : \Ty^\bullet(A))\ (x^\bullet : A^\bullet(x))\ (y^\bullet : A^\bullet(y)) \to \overline{\CS}.\Tm(1,\Id(A,x,y)) \to \SSet^{\mathsf{fib}}_n
\end{alignat*}
generated by a single constructor
\begin{alignat*}{1}
  & \refl^\bullet : \forall A\ x\ (A^\bullet : \Ty^\bullet(A))\ (x^\bullet : A^\bullet(x)) \to \Id^\bullet(A^\bullet,x^\bullet,x^\bullet,\overline{\CS}.\refl(A,x)).
\end{alignat*}

The displayed eliminator for the identity types is interpreted using the elimination principle of $\Id^\bullet$.

\begin{lem}\label{lem:id_hit_equiv_path}
  There is an equivalence $\Id^\bullet(A^\bullet,x^\bullet,y^\bullet,\refl) \simeq (x^\bullet \sim y^\bullet)$.
\end{lem}
\begin{proof}
  It suffices to prove that $(y^\bullet : A^\bullet(x)) \times \Id^\bullet(A^\bullet,x^\bullet,y^\bullet,\refl)$ is contractible.
  
  The universal property of $\Id^\bullet$ implies that
  \[ (y : \overline{\CS}.\Tm(\Gamma,A)) \times (p : \overline{\CS}.\Tm(1,\Id(A,x,y)) \times (y^\bullet : A^\bullet(x)) \times \Id^\bullet(A^\bullet,x^\bullet,y^\bullet,p) \]
  is contractible, so the result follows from the contractibility of $(y : \overline{\CS}.\Tm(\Gamma,A)) \times (p : \overline{\CS}.\Tm(1,\Id(A,x,y))$, \ie from $\overline{\CS}$ being complete.
\end{proof}

\subsubsection{Pi-types}

Take a displayed type $A^\bullet : \Ty^\bullet_n(A)$ and a family
\[ B^\bullet : \forall (a : \overline{\CS}.\Tm(1,A))\ (a^\bullet : A^\bullet(a)) \to \Ty^\bullet_n(B[a]). \]

The logical predicate over the $\Pi$-type $\Pi(A,B)$ is:
\begin{alignat*}{1}
  & \Pi^\bullet(A^\bullet,B^\bullet) \triangleq \lambda f \mapsto (\forall (a : \overline{\CS}.\Tm(1,A))\ (a^\bullet : A^\bullet(a)) \to B^\bullet(a^\bullet, \app(f,a))).
\end{alignat*}

\subsubsection{Sigma-types}

Take a displayed type $A^\bullet : \Ty^\bullet_n(A)$ and a family
\[ B^\bullet : \forall (a : \overline{\CS}.\Tm(1,A))\ (a^\bullet : A^\bullet(a)) \to \Ty^\bullet_n(B[a]). \]

The logical predicate over the $\Sigma$-type $\Sigma(A,B)$ is:
\begin{alignat*}{1}
  & \Sigma^\bullet(A^\bullet,B^\bullet) \triangleq \lambda p \mapsto (a^\bullet : A^\bullet(\fst(p))) \times (b^\bullet : B^\bullet(a^\bullet,\snd(p))).
\end{alignat*}

\subsubsection{Boolean-types}

The logical predicate over the Boolean-type $\BoolTy$ is the fibrant inductive family generated by:
\begin{alignat*}{1}
  & \BoolTy^\bullet : \overline{\CS}.\Tm(1,\BoolTy) \to \SSet^\fib_n, \\
  & \true^\bullet : \BoolTy^\bullet(\true), \\
  & \false^\bullet : \BoolTy^\bullet(\false).
\end{alignat*}

\subsubsection{$W$-types}

Take a displayed type $A^\bullet : \Ty^\bullet_n(A)$ and a family
\[ B^\bullet : \forall (a : \overline{\CS}.\Tm(1,A))\ (a^\bullet : A^\bullet(a)) \to \Ty^\bullet_n(B[a]). \]

The logical predicate over the type $W(A,B)$ is the fibrant inductive family generated by:
\begin{alignat*}{1}
  & W^\bullet : \overline{\CS}.\Tm(1,W(A,B)) \to \SSet^\fib_n, \\
  & \mathsf{sup}^\bullet : \forall a\ f \to (a^\bullet : A^\bullet(a)) \to (\forall b \to B^\bullet(a^\bullet,b) \to W^\bullet(f(b))) \to W^\bullet(\mathsf{sup}(a,f)).
\end{alignat*}

\subsubsection{Universes}

Fix a universe level $n$.

The logical predicate for the $n$-th universe is
\[ \UU^\bullet_n \triangleq \lambda A \mapsto (\overline{\MS}.\Tm(1,\El(A)) \to \SSet^{\mathsf{fib}}_n). \]
In other words, elements of $\UU^\bullet(A)$ are $n$-small logical predicates over closed terms of type $\El(A)$.

We have isomorphisms $\El^\bullet : \Tm^\bullet(\UU^\bullet_n, A) \cong \Ty^\bullet(A)$.

\subsubsection{Univalence}

Last but not least, we have to define the displayed univalence structure.

We recall the definitions of the displayed contractibility witnesses and equivalences.
\begin{alignat*}{1}
  & \isContr^\bullet(A^\bullet,c) = (x^\bullet : A^\bullet(\fst(c))) \times (\forall y\ (y^\bullet : A^\bullet(y)) \to \Id^\bullet(A^\bullet,x^\bullet,y^\bullet, \app(\snd(c), y))), \\
  & \isEquiv^\bullet(f^\bullet,e) \\
  & \quad = (\forall b\ b^\bullet \to \isContr^\bullet(\lambda (a,p) \mapsto (a^\bullet : A^\bullet(a)) \times (p^\bullet : \Id^\bullet(B^\bullet,f^\bullet(a^\bullet),b^\bullet,p)), \app(e,b))), \\
  & \Equiv^\bullet(A^\bullet,B^\bullet,f) = (f^\bullet : \forall a \to A^\bullet(a) \to B^\bullet(\app(\fst(f),a))) \times \isEquiv^\bullet(f^\bullet,\snd(f)).
\end{alignat*}
 
We start by relating these notions to cubical notions of contractibility and equivalences.
\begin{lem}\label{lem:contr_simpl}
  Let $A^\bullet : \Ty^\bullet_n(A)$ be a displayed type and $c : \overline{\CS}.\Tm(1,\isContr(A))$ be a witness of the contractibility of $A$.
 
  Then there is an equivalence
  \[ \isContr^\bullet(A^\bullet, c) \simeq (\forall (a : \overline{\CS}.\Tm(1,A)) \to \isContr(A^\bullet(a))). \]
\end{lem}
\begin{proof}
  Since $A$ is contractible in $\overline{\CS}$, its set of terms $\overline{\CS}.\Tm(1,A)$ is contractible by~\autoref{lem:complete_model_contr_internal_iff_external}.
  For any $x,y : \overline{\CS}.\Tm(1,A)$, the set $\overline{\CS}.\Tm(1,\Id_A(x,y))$ is also contractible by~\autoref{lem:complete_model_contr_internal_iff_external}.
  
  We have the following chain of equivalences:
  \begin{alignat*}{1}
    & \isContr^\bullet(A^\bullet, c) 
    \\
    & \quad \simeq (x^\bullet : A^\bullet(\fst(c))) \times (\forall y\ (y^\bullet : A^\bullet(y)) \to \Id^\bullet(A^\bullet,x^\bullet,y^\bullet, \app(\snd(c), y))) 
      \tag*{(Definition)}
    \\
    & \quad \simeq \forall (a : \overline{\CS}.\Tm(1,A)) \to (x^\bullet : A^\bullet(a)) \times (\forall (y^\bullet : A^\bullet(a)) \to \Id^\bullet(A^\bullet,x^\bullet,y^\bullet, \app(\snd(c), a))) 
      \tag*{(Contractibility of $\overline{\CS}.\Tm(1,A)$)}
    \\
    & \quad \simeq \forall (a : \overline{\CS}.\Tm(1,A)) \to (x^\bullet : A^\bullet(a)) \times (\forall (y^\bullet : A^\bullet(a)) \to \Id^\bullet(A^\bullet,x^\bullet,y^\bullet, \refl)) 
      \tag*{(Contractibility of $\overline{\CS}.\Tm(1,\Id_A(a,a))$)}
    \\
    & \quad \simeq \forall (a : \overline{\CS}.\Tm(1,A)) \to (x^\bullet : A^\bullet(a)) \times (\forall (y^\bullet : A^\bullet(a)) \to (x^\bullet \sim y^\bullet))
      \tag*{(By~\autoref{lem:id_hit_equiv_path})}
    \\
    & \quad \simeq \forall (a : \overline{\CS}.\Tm(1,A)) \to \isContr(A^\bullet(a)).
      \tag*{(Definition of $\isContr$) \qedhere}
  \end{alignat*}
\end{proof}

\begin{lem}\label{lem:equiv_simpl}
  Let $A^\bullet : \Ty^\bullet_n(A)$ and $B^\bullet : \Ty^\bullet_n(B)$ be two displayed types, along with a displayed map
  \[ f^\bullet : \forall a\ (a^\bullet : A^\bullet(a)) \to B^\bullet(\app(f,a)) \]
  and an element $e : \overline{\CS}.\Tm(1,\isEquiv(f))$.
  
  Then there is an equivalence
  \[ \isEquiv^\bullet(A^\bullet,B^\bullet, e) \simeq (\forall (a : \overline{\CS}.\Tm(1,A)) \to \isEquiv(f^\bullet_a)). \]
\end{lem}
\begin{proof}
  We have the following chain of equivalences:
  \begin{alignat*}{1}
    & \isEquiv^\bullet(A^\bullet,B^\bullet, e)
    \\
    & \quad \simeq \forall b\ b^\bullet \to \isContr^\bullet(\lambda (a,p) \mapsto (a^\bullet : A^\bullet(a)) \times (p^\bullet : \Id^\bullet(B^\bullet,f^\bullet(a^\bullet),b^\bullet,p)), \app(e,b))
      \tag*{(Definition)}
    \\
    & \quad \simeq \forall b\ b^\bullet \to \forall a\ p \to \isContr((a^\bullet : A^\bullet(a)) \times (p^\bullet : \Id^\bullet(B^\bullet,f^\bullet(a^\bullet),b^\bullet,p)))
      \tag*{(By~\autoref{lem:contr_simpl})}
    \\
    & \quad \simeq \forall a\ b^\bullet \to \isContr((a^\bullet : A^\bullet(a)) \times (p^\bullet : \Id^\bullet(B^\bullet,f^\bullet(a^\bullet),b^\bullet,\refl)))
      \tag*{(Contraction of $(b,p)$ to $(a,\refl)$)}
    \\
    & \quad \simeq \forall a\ b^\bullet \to \isContr((a^\bullet : A^\bullet(a)) \times (p^\bullet : f^\bullet(a^\bullet) \sim b^\bullet))
      \tag*{(By~\autoref{lem:id_hit_equiv_path})}
    \\
    & \quad \simeq \forall (a : \overline{\CS}.\Tm(1,A)) \to \isEquiv(f^\bullet_a).
      \tag*{(Definition of $\isEquiv$) \qedhere}
  \end{alignat*}
\end{proof}

We can now interpret univalence in $\MS^\bullet$.
Take a displayed type $A^\bullet : \Ty^\bullet_n(A)$.
We have to construct 
\[ \ua^\bullet(A^\bullet) : \isContr^\bullet(\lambda (B,E) \mapsto (B^\bullet : \Tm(1,B) \to \SSet^\fib_n) \times \Equiv^\bullet(A^\bullet,B^\bullet,E),\ua(A)). \]

By~\autoref{lem:contr_simpl}, it suffices to prove, for every $B : \Ty_n(1)$ and $E : \Equiv(A,B)$, the contractibility of
\[ (B^\bullet : \overline{\CS}.\Tm(1,B) \to \SSet^\fib_n) \times \Equiv^\bullet(A^\bullet,B^\bullet,E). \]

By~\autoref{lem:complete_model_contr_internal_iff_external} and univalence in $\overline{\CS}$, the set $(B : \overline{\CS}.\Ty_n) \times \overline{\CS}.\Tm(1,\Equiv(A,B))$ is contractible.
We can thus assume without loss of generality that $(B,E) = (A,\id_A)$.
By~\autoref{lem:equiv_simpl}, it then suffices to prove the contractibility of
\[ (B^\bullet : \overline{\CS}.\Tm(1,A) \to \SSet^\fib_n) \times (f^\bullet : \forall a \to A^\bullet(a) \to B^\bullet(a)) \times (\forall a \to \isEquiv(f^\bullet_a)). \]

We can move the quantification on $a : \overline{\CS}.\Tm(1,A)$ outside of the contractibility condition.
It then suffices to prove, for every $a$, the contractibility of
\[ (B^\bullet : \SSet^\fib_n) \times (f^\bullet : A^\bullet(a) \to B^\bullet(a)) \times \isEquiv(f^\bullet). \]
This is exactly univalence for the universe $\SSet^\fib_n$, which holds in cartesian cubical sets.

\subsection{Homotopy canonicity}

We have defined a displayed higher-order model $\MS^\bullet$ of HoTT over $\overline{\CS}$.
We can consider its displayed contextualization (sconing) $\CScone_{\MS^\bullet} \rightarrowtriangle \overline{\CS}$.
By the universal property of the model $\CS$, we obtain a section $\sem{-}$ of $\CScone_{\MS^\bullet}[i]$.

\[ \begin{tikzcd}
    &
    \CScone_{\MS^\bullet}
    \ar[d, -{Triangle[open]}]
    \\
    \CS
    \ar[r, "i"']
    \ar[ru, "{\sem{-}}"]
    &
    \overline{\CS}.
  \end{tikzcd} \]

We can now prove homotopy canonicity:
\begin{proof}[Proof of~\autoref{thm:homotopy_canonicity}]
  We have to prove that the model $1^\ast_\square(\CS)$, which is initial, satisfies homotopy canonicity.
  Let $b$ be a global element of $\CS.\Tm(1,\BoolTy)$.
  
  Applying the section $\sem{-}$ to $b$, we obtain a global element
  \[ \sem{b} : \BoolTy^\bullet(i(b)). \]
  
  By the universal property of $\BoolTy^\bullet$, we obtain a global element of
  \[ \overline{\CS}.\Tm(1,\Id(i(b), \true)) + \overline{\CS}.\Tm(1,\Id(i(b), \false)). \]

  Externally, this can be seen as an element of
  \[ 1^\ast_\square(\overline{\CS}).\Tm(1,\Id(i(b), \true)) + 1^\ast_\square(\overline{\CS}).\Tm(1,\Id(i(b), \false)). \]

  Since $1^\ast_\square(i) : 1^\ast_\square(\CS) \to 1^\ast_\square(\overline{\CS})$ is a split weak equivalence, we have a global element of
  \[ 1^\ast_\square(\CS).\Tm(1,\Id(i(b), \true)) + 1^\ast_\square(\CS).\Tm(1,\Id(i(b), \false)), \]
  as needed.
\end{proof}


\section{Future work}

In this paper we have only performed the construction of the strict Rezk completion that was needed for the proof of homotopy canonicity.
To enable their use in other applications, strict Rezk completions should be studied more abstractly in future work.

We have constructed strict Rezk completions for the generalized algebraic theories of categories and of democratic models of HoTT.
The two proofs already share a large part of their structure; this should be abstracted into general constructions for any generalized algebraic theories with a homotopy theory satisfying some conditions.

We have shown that strict Rezk completions exist in cartesian cubical sets, and that the inclusions become split weak equivalences after externalization.
It would be interesting to generalize the constructions to other presheaf models such as De Morgan cubical sets or (classically) simplicial sets.
Since we use the axiomatization of~\textcite{UnifyingCubicalModels}, our constructions are almost valid in De Morgan cubical sets, except for the fact that we use diagonal cofibrations in the proof of~\autoref{prop:cat_haswcoe_glue}.

The externalization functor $1^\ast_\square : \CcSet \to \SSet$ should also be generalized to other inverse image functors $F^\ast : \CPsh(\CC) \to \CPsh(\CA)$ such $F^\ast(\Cof) \cong \{\true,\false\}$, perhaps satisfying some other conditions.
For applications, functors of the form $\angles{\id,1_\square} : \CA \to (\CA \times \square)$ seem important.

As noted in~\autoref{rem:cof_fibration}, the strict Rezk-completion can be seen as a form of fibrant replacement, parametrized by a notion of cofibration.
Generally, any (algebraic) weak factorization system can be parametrized by a notion of cofibration.
Christian Sattler has suggested parametrizing whole homotopy theories ((semi) model structures) by a notion of cofibration.

The extension structures of a strict Rezk completion $\overline{\CM}$ of a model $\CM$ are not strictly stable under substitution: we do not have $\ext_\Tm(x[f]) = \ext_\Tm(x)[f]$ as a strict equality when $x \in \CM.\Tm(\Gamma,A)$ and $f \in \CM(\Delta,\Gamma)$.
They are however weakly stable, since contractibility is a homotopy proposition.
It would be interesting to know whether strict Rezk completion with strictly stable extension structures can be constructed.
Having strictly stable extension operations would make them available
internally to $\CPsh(\overline{\CM})$.


\printbibliography

@article{CanonicityCubical,
  author       = {Thierry Coquand and
                  Simon Huber and
                  Christian Sattler},
  title        = {Canonicity and homotopy canonicity for cubical type theory},
  journal      = {Log. Methods Comput. Sci.},
  volume       = {18},
  number       = {1},
  year         = {2022},
  url          = {https://doi.org/10.46298/lmcs-18(1:28)2022},
  doi          = {10.46298/lmcs-18(1:28)2022},
  bibsource    = {dblp computer science bibliography, https://dblp.org}
}

@article{CCHM17,
  author       = {Cyril Cohen and
                  Thierry Coquand and
                  Simon Huber and
                  Anders M{\"{o}}rtberg},
  title        = {Cubical Type Theory: {A} Constructive Interpretation of the Univalence
                  Axiom},
  journal      = {{FLAP}},
  volume       = {4},
  number       = {10},
  pages        = {3127--3170},
  year         = {2017},
  url          = {http://collegepublications.co.uk/ifcolog/?00019},
  bibsource    = {dblp computer science bibliography, https://dblp.org}
}

@article{HomotopyTheoryTTs,
  title        = {The homotopy theory of type theories},
  author       = {Krzysztof Kapulkin and Peter LeFanu Lumsdaine},
  year         = 2018,
  journal      = {Advances in Mathematics},
  volume       = 337,
  pages        = {1--38},
  doi          = {https://doi.org/10.1016/j.aim.2018.08.003},
  issn         = {0001-8708},
  url          = {https://www.sciencedirect.com/science/article/pii/S0001870818303062}
}

@article{UnivalentCategories,
  author       = {Benedikt Ahrens and
                  Krzysztof Kapulkin and
                  Michael Shulman},
  title        = {Univalent categories and the Rezk completion},
  journal      = {Math. Struct. Comput. Sci.},
  volume       = {25},
  number       = {5},
  pages        = {1010--1039},
  year         = {2015},
  url          = {https://doi.org/10.1017/S0960129514000486},
  doi          = {10.1017/S0960129514000486},
  bibsource    = {dblp computer science bibliography, https://dblp.org}
}

@misc{FoundationSyntheticAlgGeo,
      title={A Foundation for Synthetic Algebraic Geometry}, 
      author={Felix Cherubini and Thierry Coquand and Matthias Hutzler},
      year={2023},
      eprint={2307.00073},
      archivePrefix={arXiv},
      primaryClass={math.AG}
}

@inproceedings{UnifyingCubicalModels,
  author       = {Evan Cavallo and
                  Anders M{\"{o}}rtberg and
                  Andrew W. Swan},
  editor       = {Maribel Fern{\'{a}}ndez and
                  Anca Muscholl},
  title        = {Unifying Cubical Models of Univalent Type Theory},
  booktitle    = {28th {EACSL} Annual Conference on Computer Science Logic, {CSL} 2020,
                  January 13-16, 2020, Barcelona, Spain},
  series       = {LIPIcs},
  volume       = {152},
  pages        = {14:1--14:17},
  publisher    = {Schloss Dagstuhl - Leibniz-Zentrum f{\"{u}}r Informatik},
  year         = {2020},
  url          = {https://doi.org/10.4230/LIPIcs.CSL.2020.14},
  doi          = {10.4230/LIPIcs.CSL.2020.14},
  bibsource    = {dblp computer science bibliography, https://dblp.org}
}

@article{SyntaxModelsCartCTT,
  author       = {Carlo Angiuli and
                  Guillaume Brunerie and
                  Thierry Coquand and
                  Robert Harper and
                  Kuen{-}Bang Hou (Favonia) and
                  Daniel R. Licata},
  title        = {Syntax and models of Cartesian cubical type theory},
  journal      = {Math. Struct. Comput. Sci.},
  volume       = {31},
  number       = {4},
  pages        = {424--468},
  year         = {2021},
  url          = {https://doi.org/10.1017/S0960129521000347},
  doi          = {10.1017/S0960129521000347},
  bibsource    = {dblp computer science bibliography, https://dblp.org}
}

@article{HomotopicalInverseDiagrams,
  title        = {Homotopical inverse diagrams in categories with attributes},
  author       = {Krzysztof Kapulkin and Peter LeFanu Lumsdaine},
  year         = 2021,
  month        = {04},
  journal      = {Journal of Pure and Applied Algebra},
  volume       = 225,
  pages        = 106563,
  doi          = {10.1016/j.jpaa.2020.106563}
}

@inproceedings{InternalSconing,
  author       = {Rafa{\"{e}}l Bocquet and
                  Ambrus Kaposi and
                  Christian Sattler},
  editor       = {Marco Gaboardi and
                  Femke van Raamsdonk},
  title        = {For the Metatheory of Type Theory, Internal Sconing Is Enough},
  booktitle    = {8th International Conference on Formal Structures for Computation
                  and Deduction, {FSCD} 2023, July 3-6, 2023, Rome, Italy},
  series       = {LIPIcs},
  volume       = {260},
  pages        = {18:1--18:23},
  publisher    = {Schloss Dagstuhl - Leibniz-Zentrum f{\"{u}}r Informatik},
  year         = {2023},
  url          = {https://doi.org/10.4230/LIPIcs.FSCD.2023.18},
  doi          = {10.4230/LIPIcs.FSCD.2023.18},
  bibsource    = {dblp computer science bibliography, https://dblp.org}
}

@article{MarriageUnivalenceParametricity,
  author       = {Nicolas Tabareau and
                  {\'{E}}ric Tanter and
                  Matthieu Sozeau},
  title        = {The Marriage of Univalence and Parametricity},
  journal      = {J. {ACM}},
  volume       = {68},
  number       = {1},
  pages        = {5:1--5:44},
  year         = {2021},
  url          = {https://doi.org/10.1145/3429979},
  doi          = {10.1145/3429979},
  bibsource    = {dblp computer science bibliography, https://dblp.org}
}

@unpublished{SattlerHomotopyCanonicityHoTT,
  title  = {Homotopy canonicity of homotopy type theory},
  author = {Christian Sattler and
            Krzysztof Kapulkin},
  year   = {2019},
  note   = {Homotopy Type Theory 2019},
  url    = {https://hott.github.io/HoTT-2019//programme/#sattler},
}

@article{ShulmanCrisp, title={Brouwer's fixed-point theorem in real-cohesive homotopy type theory}, volume={28}, DOI={10.1017/S0960129517000147}, number={6}, journal={Mathematical Structures in Computer Science}, publisher={Cambridge University Press}, author={Michael Shulman}, year={2018}, pages={856–941}}

@article{DependentRightAdjoints,
author = {Birkedal, Lars and Clouston, Ranald and Mannaa, Bassel and Møgelberg, Rasmus and Pitts, Andrew and Spitters, Bas},
year = {2020},
month = {02},
pages = {118-138},
title = {Modal dependent type theory and dependent right adjoints},
volume = {30},
journal = {Mathematical Structures in Computer Science},
doi = {10.1017/S0960129519000197}
}

@article{KrausInftyCwfs,
  author =        {Kraus, Nicolai},
  title =         {Internal $\infty$-Categorical Models of Dependent Type Theory : Towards 2LTT Eating HoTT}, 
  year =          {2021},
  booktitle =     {Symposium on Logic in Computer Science (LICS 2021)}, 
  pages =         {1-14},
  doi =           {10.1109/LICS52264.2021.9470667}
}

@incollection{CwFsUSD,
  title={Categories with families: Unityped, simply typed, and dependently typed},
  author={Castellan, Simon and Clairambault, Pierre and Dybjer, Peter},
  booktitle={Joachim Lambek: The Interplay of Mathematics, Logic, and Linguistics},
  pages={135--180},
  year={2021},
  publisher={Springer}
}

@inproceedings{InternalTypeTheory,
  title        = {Internal Type Theory},
  author       = {Peter Dybjer},
  year         = 1995,
  booktitle    = {Types for Proofs and Programs, International Workshop TYPES'95, Torino, Italy, June 5-8, 1995, Selected Papers},
  publisher    = {Springer},
  series       = {Lecture Notes in Computer Science},
  volume       = 1158,
  pages        = {120--134},
  doi          = {10.1007/3-540-61780-9\_66},
  url          = {https://doi.org/10.1007/3-540-61780-9\_66},
  editor       = {Stefano Berardi and Mario Coppo},
  bibsource    = {dblp computer science bibliography, https://dblp.org}
}

@article{AxiomsCubicalTTInTopos,
  title={Axioms for Modelling Cubical Type Theory in a Topos},
  author={Ian Orton and Andrew M. Pitts},
  journal={Log. Methods Comput. Sci.},
  year={2016},
  volume={14},
  url={https://api.semanticscholar.org/CorpusID:4870266}
}

@inproceedings{InternalUniversesModelsHOTT,
  author       = {Daniel R. Licata and
                  Ian Orton and
                  Andrew M. Pitts and
                  Bas Spitters},
  editor       = {H{\'{e}}l{\`{e}}ne Kirchner},
  title        = {Internal Universes in Models of Homotopy Type Theory},
  booktitle    = {3rd International Conference on Formal Structures for Computation
                  and Deduction, {FSCD} 2018, July 9-12, 2018, Oxford, {UK}},
  series       = {LIPIcs},
  volume       = {108},
  pages        = {22:1--22:17},
  publisher    = {Schloss Dagstuhl - Leibniz-Zentrum f{\"{u}}r Informatik},
  year         = {2018},
  url          = {https://doi.org/10.4230/LIPIcs.FSCD.2018.22},
  doi          = {10.4230/LIPIcs.FSCD.2018.22},
  bibsource    = {dblp computer science bibliography, https://dblp.org}
}

@article{2LTT, title={Two-level type theory and applications}, DOI={10.1017/S0960129523000130}, journal={Mathematical Structures in Computer Science}, publisher={Cambridge University Press}, author={Annenkov, Danil and Capriotti, Paolo and Kraus, Nicolai and Sattler, Christian}, year={2023}, pages={1–56}}

@article {UnivalentFoundationsProject,
	title = {Univalent Foundations Project},
	journal = {a modified version of an NSF grant application},
	year = {2010},
	month = {10},
	pages = {1{\textendash}12},
	url = {http://www.math.ias.edu/vladimir/files/univalent_foundations_project.pdf},
	author = {Voevodsky, Vladimir}
}

@article{UnivalenceInverseDiagramsHomotopyCanonicity,
author = {Shulman, Michael},
year = {2014},
month = {06},
pages = {1203-1277},
title = {Univalence for inverse diagrams and homotopy canonicity},
volume = {25},
journal = {Mathematical Structures in Computer Science},
doi = {10.1017/S0960129514000565}
}

@inproceedings{NormalizationCubical,
  author       = {Jonathan Sterling and
                  Carlo Angiuli},
  title        = {Normalization for Cubical Type Theory},
  booktitle    = {36th Annual {ACM/IEEE} Symposium on Logic in Computer Science, {LICS}
                  2021, Rome, Italy, June 29 - July 2, 2021},
  pages        = {1--15},
  publisher    = {{IEEE}},
  year         = {2021},
  url          = {https://doi.org/10.1109/LICS52264.2021.9470719},
  doi          = {10.1109/LICS52264.2021.9470719},
  bibsource    = {dblp computer science bibliography, https://dblp.org}
}

@article{HuberCanonicity,
  author       = {Simon Huber},
  title        = {Canonicity for Cubical Type Theory},
  journal      = {J. Autom. Reason.},
  volume       = {63},
  number       = {2},
  pages        = {173--210},
  year         = {2019},
  url          = {https://doi.org/10.1007/s10817-018-9469-1},
  doi          = {10.1007/s10817-018-9469-1},
  bibsource    = {dblp computer science bibliography, https://dblp.org}
}

@misc{InftyTypeTheories,
  title        = {$\infty$-type theories},
  author       = {Nguyen, Hoang Kim and Uemura, Taichi},
  year         = 2022,
  publisher    = {arXiv},
  doi          = {10.48550/ARXIV.2205.00798},
  url          = {https://arxiv.org/abs/2205.00798},
  copyright    = {Creative Commons Attribution 4.0 International}
}

\end{document}